%% file: main.tex
\documentclass[preprint,12pt]{elsarticle}

\usepackage{amsmath,amssymb,amsthm,mathtools}
\usepackage{booktabs}
\usepackage{enumitem}
\usepackage{graphicx}
\usepackage[Export]{adjustbox}
\usepackage{microtype}
\usepackage{placeins}
\usepackage{tikz}
\usetikzlibrary{arrows.meta,positioning,shapes.geometric}

\journal{Journal of Computational Physics}

\newtheorem{definition}{Definition}[section]
\newtheorem{lemma}[definition]{Lemma}
\newtheorem{proposition}[definition]{Proposition}
\newtheorem{theorem}[definition]{Theorem}

\newtheorem{remark}[definition]{Remark}

\newcommand{\dd}{\,\mathrm{d}}
\newcommand{\jump}[1]{\left[\!\left[#1\right]\!\right]}
\newcommand{\mean}[1]{\left\{\!\left\{#1\right\}\!\right\}}
\newcommand{\uhat}{\widehat U}
\newcommand{\fhat}{\widehat F}
\newcommand{\vhat}{\widehat v}
\newcommand{\mhat}{\widehat m}

\newcommand{\Gset}{\mathcal G}
\newcommand{\LLF}{\mathrm{LLF}}
\newcommand{\EC}{\mathrm{EC}}

\allowdisplaybreaks
\adjustboxset*{max height=0.78\textheight}

\makeatletter
\renewenvironment{figure}[1][!htbp]
  {\@float{figure}[!htbp]}
  {\end@float}
\makeatother

\let\FloatBarrier\relax

\begin{document}

\begin{frontmatter}

\title{Entropy-Stable and Physical-Constraint-Preserving DGSEM for
Symmetry-Reduced General-Relativistic Hydrodynamics on Stationary
Spacetimes\tnoteref{dedication}}
\tnotetext[dedication]{This paper is dedicated to Peter D.~Lax.}

\author[aff1]{Guosheng Fu\corref{cor1}}
\ead{gfu@nd.edu}
\cortext[cor1]{Corresponding author}
\author[aff2]{Jian-Guo Liu}
\address[aff1]{Department of Applied and Computational Mathematics and
Statistics, University of Notre Dame, Notre Dame, IN 46556, USA}
\address[aff2]{Departments of Mathematics and Physics, Duke University,
Durham, NC 27708, USA}

\begin{abstract}
We develop an entropy-stable and physical-constraint-preserving discontinuous
Galerkin spectral element method for symmetry-reduced general-relativistic
hydrodynamics on prescribed stationary spacetimes.  Using a local orthonormal
transformation, the fluid variables are expressed in a form for which the
relativistic hydrodynamic algebra and the admissible set are independent of
the spatial metric, while the spacetime geometry enters through stationary
coefficients.  This separation allows entropy-conservative
special-relativistic fluxes to be combined with a compatible discretization of
the geometric source terms.  On affine tensor-product meshes, the resulting
DGSEM is conservative and satisfies a semidiscrete entropy inequality, while
the transformed variables provide a convex framework for physical-constraint
preservation.

For practical stabilization, we use a geometry-only causal speed that is
sufficient for both classical local Lax--Friedrichs entropy dissipation and the
physical-constraint-preserving Lax--Friedrichs splitting.  The fully discrete
method combines this stabilization with SSP Runge--Kutta time stepping,
oscillation elimination, and conservative local-orthonormal-state scaling.
Numerical experiments
cover smooth and strongly shocked special-relativistic flows, an axisymmetric
jet, stationary Michel accretion, Schwarzschild Bondi--Hoyle flow, and four
Kerr accretion cases.  The results demonstrate the designed
high-order accuracy in smooth regimes and robust performance for demanding
relativistic flows on curved stationary backgrounds.
\end{abstract}


\end{frontmatter}

\input{sections/01_introduction.tex}

\input{sections/02_weighted_model.tex}

\input{sections/03_entropy_structure.tex}

\input{sections/04_dgsem.tex}

\input{sections/05_pcp.tex}

\input{sections/06_algorithm.tex}

\input{sections/07_numerics.tex}

\input{sections/08_conclusions.tex}

\appendix

\input{appendices/app_symmetry_reduction.tex}

\input{appendices/app_local_ec.tex}

\input{appendices/app_secondary_algebra.tex}

\input{appendices/app_llf_speed_bound.tex}

\input{appendices/app_simple_zero_treatment.tex}

\section*{Acknowledgments}

The work of G.~Fu was partially supported by the National Science Foundation
under Grant DMS-2410741.

\section*{Declaration of competing interest}

The authors declare that they have no known competing financial interests or
personal relationships that could have appeared to influence the work reported
in this paper.

\section*{Data availability}

Data will be made available on request.

\section*{Declaration of generative AI and AI-assisted technologies in the
manuscript preparation process}

During the preparation of this work, the authors used ChatGPT (OpenAI) and
Codex (OpenAI) to improve the language and readability of the manuscript and
to assist with software development, respectively.  After using these tools,
the authors reviewed and edited the content as needed and take full
responsibility for the content of the published article.

\bibliographystyle{elsarticle-num}
\bibliography{references}

\end{document}

%% file: sections/01_introduction.tex
\section{Introduction}\label{sec:introduction}

General-relativistic hydrodynamics (GRHD) describes relativistic fluids on
curved spacetimes and underlies numerical models of black-hole accretion,
relativistic jets, compact objects, and stellar collapse.  In the standard
$3+1$ Valencia formulation, the fluid equations form a conservative balance
law whose fluxes and sources both depend on the prescribed spacetime geometry
\cite{BanyulsEtAl1997,Font2008}.  High-order discontinuous Galerkin (DG)
methods combine element-local high-order approximation, conservation, and
parallel scalability and have been developed extensively for nonlinear
hyperbolic conservation laws \cite{CockburnShu1998}.  Their use in
relativistic hydrodynamics has also received increasing attention
\cite{RadiceRezzolla2011,Teukolsky2016}.  The central challenge here is to
control nonlinear stability, physical admissibility, and spatially varying
geometry in a mutually compatible way.

Entropy-stable discretizations provide one framework for nonlinear stability.
Entropy-conservative two-point fluxes combined with summation-by-parts (SBP)
flux differencing lead to high-order entropy-stable finite-difference and DG
schemes \cite{Tadmor2003,CarpenterEtAl2014,ChenShu2017,GassnerWintersKopriva2016Split}, and
corresponding constructions are available for special-relativistic
hydrodynamics (SRHD) \cite{BhoriyaKumar2020,DuanTang2019,BiswasKumar2022}.
A complementary requirement is physical admissibility: rest-mass density and
pressure must remain positive and the fluid velocity must remain subluminal.
Physical-constraint-preserving (PCP) methods enforce these conditions through
convexity, Lax--Friedrichs splitting, and conservative scaling limiters
\cite{ZhangShu2010,WuTang2015,QinShuYang2016}.  For GRHD on a prescribed
stationary spacetime, the local
transformation of Wu~\cite{Wu2017} converts the metric-dependent admissible
set into a metric-independent Lorentz-type cone.  Damping-based DG procedures,
including OFDG and OEDG, provide an additional conservative mechanism for
controlling under-resolved oscillations near shocks
\cite{LiuLuShu2022,PengSunWu2025,CaoPengWu2025}, and recent GRHD calculations
demonstrate the effectiveness of combining PCP and OE stabilization in
demanding relativistic flows \cite{CaoPengWu2025}.

The present work develops a compatible entropy-stable and PCP DG spectral
element method (DGSEM) for symmetry-reduced GRHD on a prescribed stationary
spacetime.  Starting from the covariant perfect-fluid equations and the
Valencia formulation, we transform the active momentum variables to a local
orthonormal frame.  The fluid algebra then takes the standard SRHD form, while
the active spacetime geometry appears through stationary coefficients.  When
one spatial direction is suppressed, its metric factor remains in the reduced
conservative measure and source.  Writing $\widehat U$ for the local
orthonormal conservative state, with $I,J\in\{1,2\}$ indexing the active
coordinate directions, define
\[
 j_2:=\sqrt{\det(\gamma_{IJ})},
 \qquad
 W:=j_2\widehat U,
\]
for the intrinsic two-dimensional W-state, the evolved reduced state is
\[
 \mathbb W=wW=wj_2\widehat U,
\]
where $w>0$ denotes the physical reduction weight in the main analysis.

The first central ingredient is a geometry representation shared by the flux
and source discretizations.  We keep the intrinsic stationary W-form geometry
\[
 Y=\left(j_2,A_{11},A_{12},A_{22},C^1,C^2\right)^T
\]
separate from the physical reduction weight $w$ and introduce the complete
coefficients
\[
 \mathcal A=wA,
 \qquad
 \mathcal C=wC.
\]
The entropy-conservative volume flux uses arithmetic averages of these
coefficients, while the geometric source is discretized with the same SBP
operators and nodal geometry.  This matched construction produces the exact
discrete flux--source contraction required by the entropy analysis.  On
conforming affine tensor-product elements, the resulting DGSEM is conservative
and satisfies a semidiscrete entropy inequality with an entropy-dissipative
interface flux.

The second ingredient is a PCP formulation based on the metric-independent
admissible cone supplied by the transformed variables.  We derive a sufficient
forward-Euler condition for preservation of admissible cell averages and use a
conservative scaling limiter adapted to the weighted evolved state.  The
limiter preserves the conservative cell average while enforcing the density,
cone, and relative-margin constraints needed for robust relativistic
calculations.  The detailed limiter is formulated to remain consistent with
the spatially varying geometry and the variables used by the DGSEM.

A third result gives a common stabilization scale for entropy stability and
PCP.  For the classical local Lax--Friedrichs flux, the geometry-only causal
speed
\[
 a_I^{\rm PCP}=|\beta^I|+\alpha\sqrt{\gamma^{II}}
\]
not only satisfies the Lax--Friedrichs splitting required by the PCP analysis,
but also provides sufficient dissipation for the discrete entropy condition;
the proof is given in an appendix.  Accordingly, $a_I^{\rm PCP}$ is the
stability constant used in the numerical experiments.

The fully discrete method combines SSPRK$(3,3)$ time stepping with
OE--PCP stabilization.  For weighted problems, both OE and PCP are formulated
through the regular local state $\widehat U$ and mapped back conservatively to
the evolved state.  Axis cells for which $w$ has a simple zero on the axis
face use the regular treatment in \ref{app:simple-zero-treatment}.

The numerical experiments assess both accuracy and robustness from SRHD to
black-hole accretion.  They recover the designed high-order accuracy for smooth
SRHD and stationary Michel flow and demonstrate robust behavior for
multidimensional shocks, an axisymmetric relativistic jet, Schwarzschild
Bondi--Hoyle accretion, and four Kerr accretion cases.

The analysis is restricted to prescribed stationary spacetime geometry and
conforming affine tensor-product elements in the two active coordinates.  The
semidiscrete entropy theorem concerns the periodic problem; a general entropy
analysis of physical boundary closures and the fully discrete OE/PCP
post-processing is outside the present scope.  Curvilinear computational
meshes, nonconforming interfaces, evolving spacetimes, and three-dimensional
discretizations are natural extensions.

The remainder of the paper is organized as follows.  Section~\ref{sec:wform}
derives the stationary W-form, the symmetry-reduced weighted model, and its
continuous entropy structure.  Section~\ref{sec:dgsem} gives the compatible
DGSEM and semidiscrete entropy analysis.  Section~\ref{sec:pcp} develops the
PCP framework, and Section~\ref{sec:algorithm} specifies the fully discrete
discretization.  Section~\ref{sec:numerics} presents the numerical
validation, followed by the conclusions in Section~\ref{sec:conclusions}; the
appendices collect supporting algebra, the causal-LLF proof, and the
symmetry-boundary implementation details.

%% file: sections/02_weighted_model.tex
\section{GRHD model and symmetry reduction}
\label{sec:wform}

We begin with the covariant GRHD equations on a prescribed
stationary spacetime, introduce the $3+1$ Valencia balance law and its entropy
pair, and then apply the local orthonormal transformation of Wu
\cite{Wu2017} to obtain the $W$-form. We next derive the weighted
two-dimensional reduction used in the numerical analysis, including the
geometric contribution of the suppressed direction and the corresponding
entropy law. Finally, we specialize this framework to planar SRHD,
axisymmetric SRHD, axisymmetric Schwarzschild Bondi--Hoyle accretion, and
equatorial Kerr--Schild GRHD, which serve as the model problems in the
numerical results section. We use the spacetime signature $(-,+,+,+)$ and
geometrized units, $G=c=1$.  In the black-hole models, $M$ denotes the
black-hole mass.  It consequently sets both the spatial and temporal scales:
restoring physical units, one length unit $M$ is $GM_{\rm BH}/c^2$ and one
time unit $M$ is $GM_{\rm BH}/c^3$.  Greek indices $\mu,\nu,\ldots$ denote
spacetime components, Latin indices $i,k,\ell,\ldots$ denote spatial
coordinate components, and $a,b,c,\ldots$ denote components in a local
orthonormal spatial frame. Repeated indices are summed according to the
Einstein summation convention. The presentation below is intentionally
concise; for the $3+1$ decomposition of spacetime and its geometric
foundations, we refer to the original ADM formulation~\cite{ArnowittDeserMisner1962}
and the comprehensive treatment of Gourgoulhon~\cite{Gourgoulhon2007}, while
further details on the Valencia formulation and numerical GRHD can be found in
Banyuls et al.~\cite{BanyulsEtAl1997} and the review of Font~\cite{Font2008}.

\subsection{From covariant GRHD to the W-form}
\label{sec:covariant-grhd}

On a prescribed spacetime $(\mathcal M,g)$, the equations governing a perfect
relativistic fluid are \cite{Font2008}
\begin{equation}
 \nabla_\mu J^\mu=0,
 \qquad
 \nabla_\mu T^{\mu\nu}=0,
 \label{eq:covariant-grhd}
\end{equation}
where $\nabla_\mu$ denotes the Levi--Civita covariant derivative associated
with $g$.  The first equation independently imposes local conservation of
rest mass (or baryon number), whereas the second expresses local conservation
of energy and momentum.  In a coupled Einstein--matter system, the latter
follows from the contracted Bianchi identity and Einstein's field equations;
on the prescribed spacetime considered here, both are imposed as fluid
evolution equations.  Together with the four-velocity normalization and an
equation of state, they close the perfect-fluid model.

The mass current and stress--energy tensor are
\begin{equation}
 J^\mu=\rho u^\mu,
 \qquad
 T^{\mu\nu}=\rho h u^\mu u^\nu+p g^{\mu\nu},
 \qquad
 g_{\mu\nu}u^\mu u^\nu=-1.
 \label{eq:perfect-fluid}
\end{equation}
Here $g^{\mu\nu}$ is the inverse metric, $\rho>0$ is the rest-mass density,
$p>0$ is the pressure, and $u^\mu$ is the fluid four-velocity. For the
Gamma-law equation of state,
\begin{equation}
 p=(\Gamma-1)\rho\epsilon,
 \qquad
 h=1+\epsilon+\frac{p}{\rho}
  =1+\frac{\Gamma p}{(\Gamma-1)\rho},
 \qquad
 1<\Gamma\le2,
 \label{eq:gamma-law}
\end{equation}
where $\epsilon$ is the specific internal energy.

The specific entropy
\[
 s=\log p-\Gamma\log\rho
\]
satisfies $\nabla_\mu(\rho s u^\mu)=0$ for smooth solutions. Across shocks,
the second law of thermodynamics requires
\begin{equation}
 \nabla_\mu(\rho s u^\mu)\ge0
 \label{eq:covariant-entropy-inequality}
\end{equation}
in the distributional sense.

We rewrite the spacetime metric in the standard
$3+1$ (ADM) form \cite{ArnowittDeserMisner1962,Gourgoulhon2007,
BanyulsEtAl1997,Font2008}:
\label{sec:valencia}
\begin{equation}
 g_{\mu\nu}\,\dd x^\mu\dd x^\nu
 =-\alpha^2\dd t^2
 +\gamma_{ik}(\dd x^i+\beta^i\dd t)
                 (\dd x^k+\beta^k\dd t).
 \label{eq:adm-line}
\end{equation}
The lapse $\alpha$, shift $\beta^i$, and spatial metric $\gamma_{ik}$ are the
ADM geometric coefficients. In this work they are prescribed functions of
space and are stationary:
\[
 \partial_t\alpha=\partial_t\beta^i=\partial_t\gamma_{ik}=0.
\]
Set
\begin{equation}
 \gamma=\det(\gamma_{ik}),
 \qquad
 j=\sqrt{\gamma},
 \qquad
 \sqrt{-\det(g_{\mu\nu})}=\alpha j.
 \label{eq:physical-volume-density}
\end{equation}
For the Eulerian velocity $v^i$, define
\begin{equation}
 v_i=\gamma_{ik}v^k,
 \qquad
 L=(1-v^2)^{-1/2},
 \qquad
 v^2=\gamma_{ik}v^iv^k<1,
 \qquad
 \widetilde v^i=v^i-\frac{\beta^i}{\alpha}.
 \label{eq:lorentz-factor}
\end{equation}
The corresponding four-velocity is $u^0=L/\alpha$ and
$u^i=L\widetilde v^i$.

Following the Valencia formulation \cite{BanyulsEtAl1997,Font2008}, introduce
the full-energy conservative state $U=(D,S_1,S_2,S_3,E)^T$, where
\begin{equation}
 D=\rho L,
 \qquad
 S_i=\rho hL^2v_i,
 \qquad
 E=\rho hL^2-p,
 \label{eq:valencia-state}
\end{equation}
with fluxes
\begin{equation}
 F_{\rm V}^i(U)
 =
 \begin{pmatrix}
  D\widetilde v^i\\
  S_k\widetilde v^i+p\delta^i_k\\
  E\widetilde v^i+p v^i
 \end{pmatrix}.
 \label{eq:valencia-flux}
\end{equation}
In these variables, \eqref{eq:covariant-grhd} becomes
\begin{equation}
 \partial_t(jU)
 +\partial_i\!\left(\alpha jF_{\rm V}^i\right)
 =\alpha j\,\mathcal S_{\rm V},
 \label{eq:valencia-balance}
\end{equation}
where the geometric source is
\begin{equation}
 \mathcal S_{\rm V}
 =
 \begin{pmatrix}
  0\\[1mm]
  \left(\dfrac12 T^{\mu\nu}\partial_k g_{\mu\nu}\right)_{k=1}^3\\[2mm]
  T^{\mu0}\partial_\mu\alpha
  -\alpha T^{\mu\nu}\Gamma^0_{\mu\nu}
 \end{pmatrix}.
 \label{eq:valencia-source}
\end{equation}
Here
\[
\Gamma^\sigma_{\mu\nu}
=\frac{1}{2}g^{\sigma\lambda}
\left(
\partial_\mu g_{\nu\lambda}
+\partial_\nu g_{\mu\lambda}
-\partial_\lambda g_{\mu\nu}
\right)
\]
are the Christoffel symbols associated with the metric $g$.

For the Gamma-law gas, the corresponding mathematical entropy pair is
\begin{equation}
 \widehat\eta=-\frac{Ds}{\Gamma-1},
 \qquad
 \eta=j\widehat\eta,
 \qquad
 q_\eta^i=\alpha j\widetilde v^i\widehat\eta,
 \label{eq:valencia-entropy-pair}
\end{equation}
where $\widehat\eta$ is the standard convex entropy for relativistic
hydrodynamics \cite{BhoriyaKumar2020,DuanTang2019}. In Valencia variables,
the entropy law becomes
\begin{equation}
 \partial_t\eta+\partial_iq_\eta^i\le0,
 \label{eq:valencia-entropy-law}
\end{equation}
with equality for smooth solutions.

We finally transform the Valencia system to the $W$-form introduced by Wu
\cite{Wu2017}. Let $\Theta$ be the unique upper-triangular matrix with
positive diagonal satisfying
\label{sec:w-state}
\begin{equation}
 \Theta^T\Theta=\gamma^{-1},
 \qquad
 \mathcal L=\operatorname{diag}(1,\Theta,1).
 \label{eq:theta}
\end{equation}
Let $v=(v^1,v^2,v^3)^T$ and $S=(S_1,S_2,S_3)^T$. Define
\begin{equation}
 \widehat v=\Theta^{-T}v,
 \quad
 \widehat m=\Theta S=\rho hL^2\widehat v,
 \quad
 \widehat U=\mathcal L U=(D,\widehat m,E)^T,
 \quad
 W=j\widehat U.
 \label{eq:w-state}
\end{equation}
Then $|\widehat v|^2=v^2$, and the local physical flux has the standard SRHD
form
\begin{equation}
 \widehat F^a(\widehat U)=
 \begin{pmatrix}
  D\widehat v_a\\
  \widehat m_b\widehat v_a+p\delta_{ab}\\
  \widehat m_a
 \end{pmatrix},
 \qquad a=1,2,3.
 \label{eq:local-physical-flux}
\end{equation}
Define
\begin{equation}
 A_{ai}=\alpha j\Theta_{ai},
 \qquad
 C^i=j\beta^i,
 \qquad
 H^i=A_{ai}\widehat F^a-C^i\widehat U
     =\alpha j\mathcal L F_{\rm V}^i.
 \label{eq:wflux}
\end{equation}
Since the geometry is stationary, the transformed system is
\begin{equation}
 \partial_t W +\partial_iH^i=\mathcal S_W,
 \label{eq:intrinsic-w-balance}
\end{equation}
where
\begin{equation}
 \mathcal S_W
 =\alpha j\left[
    \mathcal L\mathcal S_{\rm V}
    +(\partial_i\mathcal L)F_{\rm V}^i
   \right].
 \label{eq:w-source-explicit}
\end{equation}
The second term is generated solely by the spatial variation of the local
orthonormal frame when the Valencia balance law is multiplied by
$\mathcal L$ and the spatial product rule is applied.
Finally, the entropy pair in $W$-form is
\begin{equation}
 \widehat q^a=\widehat v_a\widehat\eta,
 \qquad
 q_\eta^i=A_{ai}\widehat q^a-C^i\widehat\eta,
 \qquad
 \partial_t \eta +\partial_iq_\eta^i\le0,
 \label{eq:continuous-intrinsic-entropy}
\end{equation}
again with equality for smooth solutions.

%% file: sections/03_entropy_structure.tex
\subsection{Two-dimensional W-form reductions}
\label{sec:entropy}
\label{sec:symmetry-reduction}

Many numerical computations evolve only two active spatial coordinates after
suppressing a third.  The common algebra below covers both genuine symmetry
reductions and the infinitesimally thin equatorial Kerr restriction used
later.  In either case, the metric scale of the suppressed direction remains
in the physical volume measure and generates additional geometric source
terms.

We split the spatial coordinates as
\begin{equation}
 x^i=(x^I,x^\perp),
 \qquad I=1,2,
 \label{eq:active-suppressed-split}
\end{equation}
where $x^I$ are the active coordinates and $x^\perp$ denotes the suppressed
coordinate. Henceforth, uppercase Latin indices $I,J,\ldots$ range over the
two active directions.  The four-component model considered here assumes
\begin{equation}
 u^\perp=0,
 \qquad
 \beta^\perp=0,
 \qquad
 \gamma_{I\perp}=0,
 \qquad
 \gamma_{\perp\perp}=w^2,
 \label{eq:suppressed-symmetry}
\end{equation}
where $w>0$ is the metric scale factor associated with the suppressed
direction.  For a true symmetry reduction we additionally assume
$\partial_\perp(\cdot)=0$.  This applies to the planar and axisymmetric
models below.  The equatorial Kerr model is different: it is the standard
infinitesimally thin restriction of the full equations to $\theta=\pi/2$,
with zero polar motion and no retained polar dependence, as in
\cite{FontIbanezPapadopoulos1999}.  The same block algebra applies on the
equatorial plane, but the suppressed polar direction is not a continuous
symmetry orbit.

The spatial metric has the block form
\begin{equation}
 \gamma_{ik}^{(3)}=
 \begin{pmatrix}
  \gamma_{IJ}&0\\
  0&w^2
 \end{pmatrix},
 \qquad
 j_3:=\sqrt{\det\gamma^{(3)}}=w j_2,
 \qquad
 j_2:=\sqrt{\det\gamma_{IJ}}.
 \label{eq:suppressed-block}
\end{equation}
Thus $j_2$ is the volume density of the active two-dimensional metric, whereas
$j_3$ also includes the suppressed-direction scale $w$.

Choose the orthonormal factor consistently with the block metric,
\[
 \Theta^{(3)}=
 \begin{pmatrix}
  \Theta&0\\
  0&w^{-1}
 \end{pmatrix},
 \qquad
 \Theta^T\Theta=\gamma_{IJ}^{-1}.
\]
After removing the vanishing $\perp$-momentum component, define
\begin{equation}
 \widehat U=(D,\widehat m_1,\widehat m_2,E)^T,
 \qquad
 W:=j_2\widehat U,
 \label{eq:reduced-w-state}
\end{equation}
and, for $I=1,2$,
\begin{equation}
 H^I:=A_{aI}\widehat F^a-C^I\widehat U,
 \qquad
 A_{aI}:=\alpha j_2\Theta_{aI},
 \qquad
 C^I:=j_2\beta^I,
 \label{eq:reduced-w-flux}
\end{equation}
where $\widehat F^a$ is the four-component orthonormal flux obtained by
restricting the flux in Section~\ref{sec:w-state} to the active variables.
Since $j_3=w j_2$, the restriction of the three-dimensional W-state and its
active fluxes is simply $wW$ and $wH^I$, respectively, while the flux of the
active variables in the suppressed direction vanishes. More precisely,
\begin{equation}
 (F_{\rm V}^{\perp})_D=0,
 \qquad
 (F_{\rm V}^{\perp})_{S_I}=0,
 \qquad
 (F_{\rm V}^{\perp})_E=0,
 \qquad
 (F_{\rm V}^{\perp})_{S_\perp}=p.
 \label{eq:suppressed-direction-flux-components}
\end{equation}
Thus the suppressed-direction flux vanishes only after restriction to the
retained four-component state; the pressure flux belongs to the omitted
suppressed-momentum equation. The W-form restricted to the active state
therefore becomes
\begin{equation}
 \partial_t(wW)+\partial_I(wH^I)
 =w\mathcal S_W+\mathbb S_\perp,
 \label{eq:physical-reduced-balance}
\end{equation}
where $\mathcal S_W$ is the W-source formed from the active geometry
$(\alpha,\beta^I,\gamma_{IJ})$, and $\mathbb S_\perp$ is the weighted
geometric source induced by the suppressed direction. This reduced
conservative system is the model for which the numerical methods in the
following sections are developed.

To identify $\mathbb S_\perp$, decompose the Valencia source
\eqref{eq:valencia-source} as
\[
 \mathcal S_{\rm V}
 =\mathcal S_{\rm V}^{\rm act}+\mathcal S_{{\rm V},\perp},
\]
where $\mathcal S_{{\rm V},\perp}$ denotes the contribution from the
suppressed direction. Since $u^\perp=0$,
\[
 T^{\perp\perp}=p g^{\perp\perp}=\frac{p}{w^2}.
\]
Multiplication by the full factor $w$ gives the active momentum contribution
\begin{equation}
 w(\mathcal S_{{\rm V},\perp})_{S_I}
 =p\,\partial_I w.
 \label{eq:reduced-valencia-momentum-source}
\end{equation}
For the energy equation,
\[
 \Gamma^0_{\perp\perp}
 =-\frac12g^{0I}\partial_I(w^2),
\]
so the corresponding weighted contribution is
\begin{equation}
 w(\mathcal S_{{\rm V},\perp})_E
 =\alpha p g^{0I}\partial_I w.
 \label{eq:reduced-valencia-energy-source}
\end{equation}
After applying the active W transformation, the weighted
suppressed-direction W-source is
\begin{equation}
 (\mathbb S_\perp)_D=0,
 \;
 (\mathbb S_\perp)_{\widehat m_a}
 =\alpha j_2\Theta_{aI}p\,\partial_I w,
 \;
 (\mathbb S_\perp)_E
 =\alpha^2j_2 g^{0I}p\,\partial_I w.
 \label{eq:reduced-suppressed-source}
\end{equation}
Equivalently, define finite coefficient vectors $P_\perp^I$ by
\begin{equation}
 \mathbb S_\perp=P_\perp^I\partial_Iw.
 \label{eq:reduced-weight-source-coefficient}
\end{equation}

The entropy condition reduces with exactly the same physical measure. Indeed,
\[
 \eta^{(3)}=w\eta,
 \qquad
 q_\eta^{I,(3)}=wq_\eta^I,
 \qquad
 q_\eta^{\perp,(3)}=0.
\]
Therefore entropy-admissible weak solutions satisfy
\begin{equation}
 \partial_t(w\eta)+\partial_I(wq_\eta^I)\le0,
 \label{eq:continuous-font-entropy}
\end{equation}
with equality for smooth solutions.

\subsection{Model geometries}
The following cases provide representative examples of the active geometry and the suppressed metric
factor in \eqref{eq:physical-reduced-balance}.  The planar, cylindrical, and
Schwarzschild cases are true symmetry reductions; the Kerr case is the thin
equatorial restriction described above.

\subsubsection{Planar special relativistic hydrodynamics.}
\label{sec:planar-srhd-model}
The simplest member of \eqref{eq:physical-reduced-balance} is the usual
planar SRHD system, obtained from three-dimensional Minkowski space by
suppressing a translationally invariant direction. In Cartesian active
coordinates $(x,y)$,
\begin{equation}
 \alpha=1,
 \qquad
 \beta^I=0,
 \qquad
 \gamma_{IJ}=\delta_{IJ},
 \qquad
 j_2=1,
 \qquad
 w=1.
 \label{eq:planar-srhd-reduction}
\end{equation}
Hence $\Theta=I_2$, $A_{aI}=\delta_{aI}$, and $C^I=0$. Therefore
\begin{equation}
 W=\widehat U,
 \qquad
 H^x=\widehat F^1,
 \qquad
 H^y=\widehat F^2,
 \qquad
 \mathcal S_W=\mathbb S_\perp=0,
 \label{eq:planar-srhd-wform}
\end{equation}
and the reduced system is
\begin{equation}
 \partial_t W+\partial_x\widehat F^1+\partial_y\widehat F^2=0.
 \label{eq:planar-srhd-balance}
\end{equation}
Such multidimensional SRHD models are standard benchmarks for high-order
physical-constraint-preserving relativistic schemes; see, e.g.,
\cite{WuTang2015,QinShuYang2016}.

\subsubsection{Axisymmetric SRHD with zero azimuthal velocity.}
\label{sec:axisymmetric-srhd-model}
A nontrivial reduction is obtained by writing flat three-dimensional space in
cylindrical coordinates $(r,z,\phi)$,
\begin{equation}
 \dd\ell^2=\dd r^2+\dd z^2+r^2\dd\phi^2,
 \label{eq:cylindrical-flat-metric}
\end{equation}
and imposing axisymmetry with zero azimuthal velocity. The active metric on
$(r,z)$ is Euclidean, so
\begin{equation}
 \alpha=1,
 \qquad
 \beta^I=0,
 \qquad
 \gamma_{IJ}=\delta_{IJ},
 \qquad
 j_2=1,
 \qquad
 w=r.
 \label{eq:axisymmetric-radial-weight}
\end{equation}
Thus $\Theta=I_2$, $W=\widehat U$, $H^r=\widehat F^r$, and
$H^z=\widehat F^z$. The active W-source vanishes, while the direct weighted
suppressed source is
\begin{equation}
 \mathbb S_\perp
 =\begin{pmatrix}
   0, p, 0, 0
  \end{pmatrix}^T.
 \label{eq:axisymmetric-srhd-source}
\end{equation}
Consequently,
\begin{equation}
 \partial_t(rW)
 +\partial_r(r\widehat F^r)
 +\partial_z(r\widehat F^z)
 =
 \begin{pmatrix}
  0, p, 0, 0
 \end{pmatrix}^T.
 \label{eq:axisymmetric-srhd-balance}
\end{equation}
The continuous entropy condition follows directly from
\eqref{eq:continuous-font-entropy}. Axisymmetric relativistic hydrodynamics
in cylindrical coordinates is a standard reduced setting; see
\cite{MartiEtAl1997,QinShuYang2016,CaoPengWu2025}.
The coordinate singularity at the symmetry axis $r=0$
 requires special attention in the design of a robust numerical method.

\subsubsection{Axisymmetric Schwarzschild Bondi--Hoyle accretion.}
\label{sec:schwarzschild-bondi-hoyle}

A genuinely two-dimensional GRHD model is axisymmetric
Bondi--Hoyle accretion of an ideal-gas wind onto a fixed Schwarzschild black
hole. Horizon-penetrating Eddington--Finkelstein slices are convenient because
the computational domain may extend through the event horizon
\cite{LoraClavijoGuzman2013}. We use active coordinates $(r,\theta)$ and
suppress the azimuthal direction $\phi$, with $u^\phi=0$. The Schwarzschild
ADM data are
\begin{equation}
 \mathcal K(r):=1+\frac{2M}{r},
 \qquad
 \alpha=\mathcal K^{-1/2},
 \qquad
 \beta^r=\frac{2M}{r+2M},
 \qquad
 \beta^\theta=0,
 \label{eq:schwarzschild-bh-lapse-shift}
\end{equation}
and
\begin{equation}
 \gamma_{IJ}
 =\begin{pmatrix}
   \mathcal K&0\\
   0&r^2
  \end{pmatrix},
 \qquad
 w=r\sin\theta,
 \qquad I,J\in\{r,\theta\}.
 \label{eq:schwarzschild-bh-spatial-metric}
\end{equation}
Thus
\begin{equation}
 j_2=r\sqrt{\mathcal K},
 \qquad
 \alpha j_2=r,
 \qquad
 \Theta=
 \begin{pmatrix}
  \mathcal K^{-1/2}&0\\
  0&r^{-1}
 \end{pmatrix}.
 \label{eq:schwarzschild-bh-frame}
\end{equation}
The W-form geometry coefficients are therefore
\begin{equation}
 (A_{aI})
 =\begin{pmatrix}
  r/\sqrt{\mathcal K}&0\\
  0&1
 \end{pmatrix},
 \qquad
 C^r=2M\alpha,
 \qquad
 C^\theta=0.
 \label{eq:schwarzschild-bh-w-geometry}
\end{equation}
Hence
\begin{equation}
 W=r\sqrt{\mathcal K}\,\widehat U,
 \qquad
 H^r=\frac{r}{\sqrt{\mathcal K}}\widehat F^1-2M\alpha\widehat U,
 \qquad
 H^\theta=\widehat F^2.
 \label{eq:schwarzschild-bh-state-flux}
\end{equation}
All active geometry coefficients depend only on $r$. Hence the intrinsic
source is obtained from the W-source formula of Section~\ref{sec:w-state} as
\begin{equation}
 \mathcal S_W
 =\alpha j_2\left[
   \mathcal L\mathcal S_{\rm V}^{\rm act}
   +(\partial_r\mathcal L)F_{\rm V}^r
  \right],
 \qquad
 \mathcal L=\operatorname{diag}(1,\Theta,1).
 \label{eq:black-hole-active-source}
\end{equation}
Its explicit expression for the Schwarzschild geometry is collected in
\ref{app:sb-rphi-details}.

The weighted source in the suppressed azimuthal direction is
\begin{equation}
 \mathbb S_\perp
 =\begin{pmatrix}
  0, \;
  r\sin\theta\,p/\sqrt{\mathcal K},\;
  r\cos\theta\,p,\;
  2M\alpha p\sin\theta
 \end{pmatrix}^T.
 \label{eq:schwarzschild-bh-suppressed-source}
\end{equation}
The reduced Schwarzschild system is therefore
\begin{equation}
 \partial_t(r\sin\theta\,W)
 +\partial_r(r\sin\theta\,H^r)
 +\partial_\theta(r\sin\theta\,H^\theta)
 =r\sin\theta\,\mathcal S_W+\mathbb S_\perp,
 \label{eq:schwarzschild-bh-balance}
\end{equation}
Its continuous entropy condition follows from
\eqref{eq:continuous-font-entropy} with $w=r\sin\theta$.  For a supersonic
asymptotic wind this model develops the characteristic downstream
shock cone of relativistic Bondi--Hoyle accretion
\cite{LoraClavijoGuzman2013}.
The polar axes $\theta=0,\pi$, where $\sin\theta=0$,
require special care in the numerical treatment.

\subsubsection{Equatorial Kerr--Schild restriction.}
\label{sec:kerr-reduced-models}
Black-hole accretion on prescribed Kerr backgrounds is a standard GRHD model
problem \cite{Font2008}. Following the infinitesimally thin
equatorial setup of \cite{FontIbanezPapadopoulos1999}, we restrict the
three-dimensional system to $\theta=\pi/2$ with zero polar velocity. The
active coordinates are the Kerr--Schild coordinates $(r,\widetilde\phi)$,
and the suppressed polar metric scale is
\begin{equation}
 w=r.
 \label{eq:font-radial-weight}
\end{equation}
Let $M>0$ be the black-hole mass, $|a|\le M$ the spin parameter, and define
\begin{equation}
 \mathcal K(r):=1+\frac{2M}{r},
 \qquad
 \mathcal B(r):=r^2+a^2\mathcal K(r).
 \label{eq:kerr-rphi-scalars}
\end{equation}
The equatorial ADM data are
\begin{equation}
 \alpha=\mathcal K^{-1/2},
 \qquad
 \beta^r=\frac{2M}{r+2M},
 \qquad
 \beta^{\widetilde\phi}=0,
 \label{eq:kerr-rphi-lapse-shift}
\end{equation}
and
\begin{equation}
 \gamma_{IJ}=
 \begin{pmatrix}
  \mathcal K&-a\mathcal K\\
  -a\mathcal K&\mathcal B
 \end{pmatrix},
 \qquad I,J\in\{r,\widetilde\phi\}.
 \label{eq:kerr-rphi-spatial-metric}
\end{equation}
Since
\[
 \det\gamma_{IJ}
 =\mathcal K\mathcal B-a^2\mathcal K^2
 =r^2\mathcal K=r(r+2M),
\]
we have
\begin{equation}
 \gamma^{IJ}=
 \begin{pmatrix}
  \mathcal K^{-1}+a^2/r^2&a/r^2\\
  a/r^2&1/r^2
 \end{pmatrix},
 \qquad
 j_2=\sqrt{r(r+2M)},
 \qquad
 \alpha j_2=r.
 \label{eq:kerr-rphi-inverse-volume}
\end{equation}
Since $w=r$ on the equatorial plane,
\begin{equation}
 wj_2=r^2\sqrt{\mathcal K}
 =\sqrt{\det(\gamma_{ij})}\big|_{\theta=\pi/2}.
 \label{eq:kerr-full-volume-consistency}
\end{equation}
Thus the weight restores exactly the full three-dimensional spatial volume
factor on the equatorial plane.
The upper-triangular orthonormal factor is
\begin{equation}
 \Theta=
 \begin{pmatrix}
  \dfrac{\sqrt{\mathcal B/\mathcal K}}{r}
  &\dfrac{a\sqrt{\mathcal K/\mathcal B}}{r}\\[3mm]
  0&\dfrac1{\sqrt{\mathcal B}}
 \end{pmatrix},
 \qquad
 \Theta^T\Theta=\gamma^{-1}.
 \label{eq:kerr-rphi-theta}
\end{equation}
Consequently, the W-form geometry coefficients are
\begin{equation}
 (A_{aI})
 =r\Theta
 =\begin{pmatrix}
  \sqrt{\mathcal B/\mathcal K}
  &a\sqrt{\mathcal K/\mathcal B}\\[2mm]
  0&r/\sqrt{\mathcal B}
 \end{pmatrix},
 \qquad
 C^r=2M\alpha,
 \qquad
 C^{\widetilde\phi}=0.
 \label{eq:kerr-rphi-w-geometry}
\end{equation}
Thus
\begin{equation}
 W=\sqrt{r(r+2M)}\,\widehat U,
 \label{eq:kerr-rphi-state}
\end{equation}
and the two active fluxes are
\begin{align}
 H^r
 &=\sqrt{\frac{\mathcal B}{\mathcal K}}\,\widehat F^1
   -2M\alpha\widehat U,
 \label{eq:kerr-rphi-radial-flux}\\
 H^{\widetilde\phi}
 &=a\sqrt{\frac{\mathcal K}{\mathcal B}}\,\widehat F^1
   +\frac{r}{\sqrt{\mathcal B}}\,\widehat F^2.
 \label{eq:kerr-rphi-azimuthal-flux}
\end{align}
All geometric coefficients depend only on $r$. Hence the intrinsic active
source has the shared transformed form
\eqref{eq:black-hole-active-source}.  The explicit active source
$S_{\rm V}^{\rm act}$ and radial frame derivatives are collected in
\ref{app:kerr-rphi-details}.

For the suppressed polar direction, $w=r$ and $\partial_rw=1$. Using
\eqref{eq:reduced-suppressed-source}, together with $g^{0r}=2M/r$, gives
\begin{equation}
 \mathbb S_\perp
 =\begin{pmatrix}
  0,
  p\sqrt{\dfrac{\mathcal B}{\mathcal K}},
  0,
  2M\alpha p
 \end{pmatrix}^T.
 \label{eq:kerr-rphi-suppressed-source}
\end{equation}
The complete equatorial Kerr model is therefore
\begin{equation}
 \partial_t(rW)
 +\partial_r(rH^r)
 +\partial_{\widetilde\phi}(rH^{\widetilde\phi})
 =r\mathcal S_W+\mathbb S_\perp,
 \label{eq:rphi-reduced-balance}
\end{equation}
Its entropy condition follows from
\eqref{eq:continuous-font-entropy} with $w=r$.

\begin{remark}[Scope and vanishing metric scales]
The four-component model requires $u^\perp=0$; flows with nonzero suppressed
velocity require an enlarged state and lie outside the present scope.  The
analysis above is carried out at points where $w>0$.  In genuine symmetry
reductions, however, the orbit scale may vanish on a symmetry axis, as in the
cylindrical and Schwarzschild examples considered above.  Such zeros are
handled separately in the numerical method.  For a simple zero of $w$, the
weighted state has the regular boundary value $\mathbb W_\star=0$, while the
finite intrinsic state is recovered from the neighboring polynomial through a
discrete l'Hopital construction; see \ref{app:simple-zero-treatment}.
\end{remark}

%% file: sections/04_dgsem.tex
\section{Entropy-stable DGSEM on affine tensor-product meshes}
\label{sec:dgsem}

We discretize the symmetry-reduced GRHD system
\eqref{eq:physical-reduced-balance} on conforming affine tensor-product
elements in the two active coordinates.  The prescribed geometry is
stationary, and throughout the entropy analysis we assume $w>0$ at all
solution nodes.
Axis cells containing a boundary where $w=0$
require a separate regular treatment in the implementation; see \ref{app:simple-zero-treatment}.

\subsection{Numerical conservative form and entropy structure}
\label{sec:dg-weighted-entropy-structure}

Define
\begin{equation}
 \mathbb W:=wW,
 \qquad
 \mathbb H^I:=wH^I,
 \qquad
 \mathbb S:=w\mathcal S_W+\mathbb S_\perp,
 \label{eq:dg-weighted-variables}
\end{equation}
so that
\begin{equation}
 \partial_t\mathbb W+\partial_I\mathbb H^I=\mathbb S.
 \label{eq:dg-numerical-balance}
\end{equation}
With
\begin{equation}
 \mathcal A_{aI}:=wA_{aI},
 \qquad
 \mathcal C^I:=wC^I,
 \label{eq:dg-complete-coefficients}
\end{equation}
the weighted flux is
\begin{equation}
 \mathbb H^I
 =\mathcal A_{aI}\fhat^a-\mathcal C^I\uhat,
 \label{eq:dg-weighted-flux}
\end{equation}
and, at a node with $w>0$,
\begin{equation}
 \uhat=\frac{\mathbb W}{wj_2}.
 \label{eq:dg-positive-weight-local-state}
\end{equation}

The corresponding entropy pair is
\begin{equation}
 \eta_w:=w\eta=wj_2\widehat\eta,
 \qquad
 q_{\eta,w}^I:=wq_\eta^I.
 \label{eq:dg-weighted-entropy-pair}
\end{equation}
Differentiation at fixed geometry gives the local SRHD entropy variables
\begin{equation}
 V:=\nabla_{\mathbb W}\eta_w
 =\begin{pmatrix}
 \dfrac{\Gamma-s}{\Gamma-1}+\zeta\\
 \zeta L\vhat_a\\
 -\zeta L
 \end{pmatrix},
 \qquad
 \zeta:=\frac{\rho}{p}.
 \label{eq:entropy-variables}
\end{equation}
Here $\nabla_{\mathbb W}$ denotes the ordinary gradient with respect to the
state variables $\mathbb W$; it is distinct from the spacetime covariant
derivative $\nabla_\mu$ in~\eqref{eq:covariant-grhd}.

Define the local entropy potentials
\begin{equation}
 \Phi:=V^T\uhat-\widehat\eta=D,
 \qquad
 \psi_a:=V^T\fhat^a-\widehat q^a=D\vhat_a,
 \qquad
 \widehat q^a=\vhat_a\widehat\eta.
 \label{eq:dg-local-potentials}
\end{equation}
Then
\begin{equation}
 q_{\eta,w}^I
 =\mathcal A_{aI}\widehat q^a-\mathcal C^I\widehat\eta,
 \qquad
 \Psi_w^I:=V^T\mathbb H^I-q_{\eta,w}^I
 =\mathcal A_{aI}\psi_a-\mathcal C^I\Phi.
 \label{eq:dg-weighted-entropy-flux}
\end{equation}

\begin{lemma}[Source contraction]
\label{lem:dg-continuous-source-contraction}
For every smooth admissible state at a point with $w>0$,
\begin{equation}
 V^T\mathbb S
 =\psi_a\,\partial_I\mathcal A_{aI}
  -\Phi\,\partial_I\mathcal C^I.
 \label{eq:dg-continuous-source-contraction}
\end{equation}
\end{lemma}
\begin{proof}
  Contract the smooth reduced balance law with the entropy variables $V$.
  Using
  \[
  d\widehat\eta = V^T d\uhat,
  \qquad
  d\widehat q^a = V^T d\fhat^a,
  \]
  the terms involving derivatives of the fluid state combine into the smooth
  entropy balance.  The remaining contributions come from spatial variation of
  the prescribed geometry coefficients.  Substituting
  \eqref{eq:dg-weighted-flux} and
  \eqref{eq:dg-weighted-entropy-flux}, these terms reduce to
  \[
  \psi_a\,\partial_I\mathcal A_{aI}
  -\Phi\,\partial_I\mathcal C^I.
  \]
  Hence
  \[
  V^T\mathbb S
  =
  \psi_a\,\partial_I\mathcal A_{aI}
  -\Phi\,\partial_I\mathcal C^I,
  \]
  which proves \eqref{eq:dg-continuous-source-contraction}.
  \end{proof}

For the numerical source, keep the intrinsic geometry below separate from
$w$; its one-to-one map to the active ADM data and the corresponding
directional differentials are detailed in
\ref{app:geometry-state-map}:
\begin{equation}
 Y:=\left(j_2,A_{11},A_{12},A_{22},C^1,C^2\right)^T
 \label{eq:dg-y-geometry-state}
\end{equation}
and write
\begin{equation}
 \mathcal S_W
 =B_{j_2}^I\partial_I j_2
 +B_{A_{aJ}}^I\partial_IA_{aJ}
 +B_{C^J}^I\partial_IC^J.
 \label{eq:dg-intrinsic-source-decomposition}
\end{equation}
Using $\mathbb S_\perp=P_\perp^I\partial_Iw$ from
\eqref{eq:reduced-weight-source-coefficient}, an algebraic product-rule
rearrangement gives
\begin{equation}
 \mathbb S
 =wB_{j_2}^I\partial_I j_2
 +B_{A_{aJ}}^I\partial_I(wA_{aJ})
 +B_{C^J}^I\partial_I(wC^J)
 +R_w^I\partial_Iw,
 \label{eq:dg-source-decomposition}
\end{equation}
where
\begin{equation}
 R_w^I
 :=P_\perp^I-B_{A_{aJ}}^IA_{aJ}-B_{C^J}^IC^J.
 \label{eq:dg-weight-residual}
\end{equation}
The corresponding entropy contractions are
\begin{equation}
 \begin{aligned}
 V^TB_{j_2}^I&=0,
 &\qquad
 V^TB_{A_{aJ}}^I&=\delta_{IJ}\psi_a,
 &\qquad
 V^TB_{C^J}^I&=-\delta_{IJ}\Phi,\\
 V^TP_\perp^I&=A_{aI}\psi_a-C^I\Phi,
 &\qquad
 V^TR_w^I&=0.
 \end{aligned}
 \label{eq:dg-source-coefficient-contractions}
\end{equation}
Hence only the derivatives of the complete coefficients $wA$ and $wC$
contribute to $V^T\mathbb S$; the $j_2$ and residual $w$ terms are
entropy-orthogonal.  This is the source form used in the DGSEM below.

\subsection{Single-element SBP-DGSEM}
\label{sec:dg-affine-elements}
\label{sec:dg-volume-flux}
\label{sec:dg-semidiscrete}

Let
\begin{equation}
 K=[x^1_L,x^1_R]\times[x^2_L,x^2_R],
 \qquad
 \Delta x_I=x^I_R-x^I_L,
 \label{eq:dg-affine-element}
\end{equation}
be one rectangular element, mapped from $\widehat K=[-1,1]^2$ by
\begin{equation}
 x^1=x^1_c+\frac{\Delta x_1}{2}\xi,
 \qquad
 x^2=x^2_c+\frac{\Delta x_2}{2}\eta.
 \label{eq:dg-affine-map}
\end{equation}
Writing $\xi^1:=\xi$, $\xi^2:=\eta$, set
\begin{equation}
 J_K=\frac{\Delta x_1\Delta x_2}{4},
 \qquad
 \kappa_I:=\frac{2}{\Delta x_I}.
 \label{eq:dg-affine-jacobian}
\end{equation}
Since $J_K$ is constant, \eqref{eq:dg-numerical-balance} becomes
\begin{equation}
 \partial_t\mathbb W
 +\sum_{I=1}^2 \kappa_I\partial_{\xi^I}\mathbb H^I
 =\mathbb S.
 \label{eq:dg-reference-balance}
\end{equation}

Let $\{\xi_i,\omega_i\}_{i=0}^N$ be the $(N+1)$ GLL nodes and positive
quadrature weights, and let $\ell_i$ be the associated Lagrange basis with
$\ell_i(\xi_j)=\delta_{ij}$.  The numerical solution and the prescribed
geometry are represented by
\begin{equation}
 \mathbb W_h(\xi,\eta,t)
 =\sum_{i,j=0}^N\mathbb W_{ij}(t)\ell_i(\xi)\ell_j(\eta),
 \qquad
 Y_h(\xi,\eta)
 =\sum_{i,j=0}^NY_{ij}\ell_i(\xi)\ell_j(\eta),
 \label{eq:dg-nodal-interpolation}
\end{equation}
and use the same nodal interpolation for $w$, $\mathcal A=wA$, and
$\mathcal C=wC$.
The one-dimensional mass, differentiation, and SBP matrices are
\begin{equation}
 \mathsf M=\operatorname{diag}(\omega_0,\ldots,\omega_N),
 \qquad
 \mathsf D_{ij}=\ell_j'(\xi_i),
 \qquad
 \mathsf Q:=\mathsf M\mathsf D,
 \label{eq:dg-mdq}
\end{equation}
and satisfy \cite{GassnerWintersKopriva2016Split}
\begin{equation}
 \mathsf Q+\mathsf Q^T=\mathsf B,
 \qquad
 \mathsf B=\operatorname{diag}(-1,0,\ldots,0,1),
 \qquad
 \mathsf D\boldsymbol 1=0.
 \label{eq:dg-sbp}
\end{equation}
For a nodal field $Z_{ij}$, set
\begin{equation}
 (\mathsf D_1Z)_{ij}
 :=\sum_{m=0}^N\mathsf D_{im}Z_{mj},
 \qquad
 (\mathsf D_2Z)_{ij}
 :=\sum_{m=0}^N\mathsf D_{jm}Z_{im}.
 \label{eq:dg-reference-derivatives}
\end{equation}

\paragraph{Entropy-conservative volume flux and compatible source}
Let $p=(i,j)$ denote a generic tensor-product node.  At each node, recover the
local state $\uhat_p$ from \eqref{eq:dg-positive-weight-local-state} and
evaluate the corresponding entropy variables $V_p$ and entropy potentials
$\Phi_p$ and $\psi_{a,p}$ using \eqref{eq:entropy-variables} and
\eqref{eq:dg-local-potentials}.  For any two-point quantity $z$, write
$\jump{z}:=z_R-z_L$ and $\mean{z}:=(z_L+z_R)/2$.
Choose symmetric, consistent SRHD two-point functions
$\uhat_{LR}^{\#}$ and $\fhat_{LR}^{a,\EC}$ satisfying
\begin{equation}
 \jump{V}^T\uhat_{LR}^{\#}=\jump{\Phi},
 \qquad
 \jump{V}^T\fhat_{LR}^{a,\EC}=\jump{\psi_a}.
 \label{eq:local-tadmor}
\end{equation}
For the spatial contribution, we use the SRHD entropy-conservative flux of
Duan and Tang \cite{DuanTang2019}, rewritten in the local orthonormal
variables.  The temporal entropy-conservative state $\uhat_{LR}^{\#}$ is
needed for the shift contribution.  Closed formulas for both two-point
objects are given in \ref{app:local-ec}.

For two nodes $p,q$ on the same nodal line in direction $I$, define
\begin{equation}
 \mathcal H_{pq}^{I,\EC}
 :=\mean{\mathcal A_{aI}}_{pq}\fhat_{pq}^{a,\EC}
   -\mean{\mathcal C^I}_{pq}\uhat_{pq}^{\#}.
 \label{eq:dg-ec-volume-flux}
\end{equation}
Then
\begin{equation}
 \jump{V}_{pq}^T\mathcal H_{pq}^{I,\EC}
 =\mean{\mathcal A_{aI}}_{pq}\jump{\psi_a}_{pq}
  -\mean{\mathcal C^I}_{pq}\jump{\Phi}_{pq}.
 \label{eq:dg-coordinate-tadmor}
\end{equation}
Here the arithmetic mean is applied directly to the complete coefficients
$\mathcal A=wA$ and $\mathcal C=wC$, rather than to their factors separately.

The compatible nodal source is
\begin{align}
 \mathbb S_{ij}^h
 :={}&\sum_{I=1}^2\kappa_I\bigl[
 w_{ij}B_{j_2,ij}^I(\mathsf D_Ij_2)_{ij}
 +B_{A_{aJ},ij}^I\bigl(\mathsf D_I(wA_{aJ})\bigr)_{ij}\notag\\
 &\hspace{34mm}
 +B_{C^J,ij}^I\bigl(\mathsf D_I(wC^J)\bigr)_{ij}
 +R_{w,ij}^I(\mathsf D_Iw)_{ij}\bigr],
 \label{eq:dg-source-discretization}
\end{align}
which satisfies the discrete entropy contraction
according to \eqref{eq:dg-source-coefficient-contractions},
\begin{equation}
 V_{ij}^T\mathbb S_{ij}^h
 =\sum_{I=1}^2\kappa_I\left[
 \psi_{a,ij}\bigl(\mathsf D_I(wA_{aI})\bigr)_{ij}
 -\Phi_{ij}\bigl(\mathsf D_I(wC^I)\bigr)_{ij}
 \right].
 \label{eq:dg-discrete-source-contraction}
\end{equation}

\paragraph{LLF surface flux and semidiscrete scheme}
For each face $f\subset\partial K$, let $I(f)\in\{1,2\}$ and
$\varepsilon_{K,f}\in\{-1,1\}$ specify its outward coordinate normal:
\begin{equation}
 n_{K,f}=\varepsilon_{K,f}e_{I(f)},
 \qquad
 J_f:=J_K\kappa_{I(f)}.
 \label{eq:dg-face-orientation}
\end{equation}
At a face GLL node, superscripts $-$ and $+$ denote the interior and exterior
traces relative to $K$.  Define the normal physical flux and entropy potential
by
\begin{equation}
 \mathbb H_{n,K,f}^{\pm}:=n_{K,f,I}\mathbb H^{I,\pm},
 \qquad
 \Psi_{w,n,K,f}^{\pm}:=n_{K,f,I}\Psi_w^{I,\pm}.
 \label{eq:dg-interface-normal-traces}
\end{equation}
The surface flux is the classical local Lax--Friedrichs (LLF) flux
\begin{equation}
 \widehat{\mathbb H}_{n,K,f}^{\LLF}
 :=\frac12\left(\mathbb H_{n,K,f}^-+\mathbb H_{n,K,f}^+\right)
   -\frac{a_f}{2}\left(\mathbb W^+-\mathbb W^-\right),
 \label{eq:dg-classical-llf}
\end{equation}
with the geometry-only causal speed
\begin{equation}
 a_f=a_{n_{K,f}}^{\rm PCP}
 :=|\beta^I n_{K,f,I}|
   +\alpha\sqrt{\gamma^{IJ}n_{K,f,I}n_{K,f,J}}.
 \label{eq:dg-causal-normal-speed}
\end{equation}
The coordinate flux used in the strong-form boundary correction is
\begin{equation}
 \widehat{\mathbb H}_{ij}^{I,\LLF}
 :=\varepsilon_{K,f}\widehat{\mathbb H}_{n,K,f}^{\LLF},
 \qquad (i,j)\in f,\quad I=I(f).
 \label{eq:dg-coordinate-llf}
\end{equation}
Combining this face flux with the entropy-conservative volume flux and the
compatible source gives the single-element semidiscretization
\begin{align}
 \dot{\mathbb W}_{ij}
 &+2\kappa_1\sum_{m=0}^N\mathsf D_{im}
   \mathcal H_{(i,j),(m,j)}^{1,\EC}
  +2\kappa_2\sum_{m=0}^N\mathsf D_{jm}
   \mathcal H_{(i,j),(i,m)}^{2,\EC}\notag\\
 &+\frac{\kappa_1\mathsf B_{ii}}{\omega_i}
   \left(\widehat{\mathbb H}_{ij}^{1,\LLF}
        -\mathbb H_{ij}^{\,1}\right)
  +\frac{\kappa_2\mathsf B_{jj}}{\omega_j}
   \left(\widehat{\mathbb H}_{ij}^{2,\LLF}
        -\mathbb H_{ij}^{\,2}\right)
 =\mathbb S_{ij}^h.
 \label{eq:dgsem-affine}
\end{align}

\paragraph{Element conservation and entropy balance}
For a nodal trace $Z$ on $f$, set
\begin{equation}
 \langle Z\rangle_{f,N}
 :=J_f\sum_{q=0}^N\omega_q Z_{f,q}.
 \label{eq:dg-face-quadratures}
\end{equation}
We further define the numerical entropy flux
\begin{equation}
 \widehat{\mathcal Q}_{n,K,f}
 :=(V^-)^T\widehat{\mathbb H}_{n,K,f}^{\LLF}
   -\Psi_{w,n,K,f}^-.
 \label{eq:dg-numerical-entropy-flux}
\end{equation}

\begin{proposition}[Single-element conservation and entropy balance]
\label{lem:dg-single-element-balance}
Define
\begin{equation}
 \mathcal U_K
 :=J_K\sum_{i,j=0}^N\omega_i\omega_j\mathbb W_{ij},
 \qquad
 \mathcal E_K
 :=J_K\sum_{i,j=0}^N\omega_i\omega_j(\eta_w)_{ij}.
 \label{eq:dg-element-functionals}
\end{equation}
The scheme \eqref{eq:dgsem-affine} satisfies the conservative balance
\begin{equation}
 \frac{d\mathcal U_K}{dt}
 +\sum_{f\subset\partial K}
   \left\langle\widehat{\mathbb H}_{n,K,f}^{\LLF}\right\rangle_{f,N}
 =J_K\sum_{i,j=0}^N\omega_i\omega_j\mathbb S_{ij}^h
 \label{eq:dg-element-balance}
\end{equation}
and the entropy balance
\begin{equation}
 \frac{d\mathcal E_K}{dt}
 +\sum_{f\subset\partial K}
   \left\langle\widehat{\mathcal Q}_{n,K,f}\right\rangle_{f,N}=0.
 \label{eq:dg-element-entropy}
\end{equation}
\end{proposition}
\begin{proof}
Summing \eqref{eq:dgsem-affine} with the tensor-product GLL weights and using
$\mathsf Q+\mathsf Q^T=\mathsf B$ gives \eqref{eq:dg-element-balance}.  Since $w_{ij}>0$,
$d(\eta_w)_{ij}/dt=V_{ij}^T\dot{\mathbb W}_{ij}$.  Contract the nodal
equations with $V_{ij}$ and apply \eqref{eq:dg-coordinate-tadmor}.  Along each
nodal line,
\[
 \sum_{p,q=0}^N\mathsf Q_{pq}\mean a_{pq}(b_p-b_q)
 =\sum_{p=0}^N\omega_pb_p(\mathsf D a)_p
  -\sum_{p=0}^N\mathsf B_{pp}a_pb_p.
\]
Equation~\eqref{eq:dg-discrete-source-contraction} cancels the remaining
interior source contraction, yielding \eqref{eq:dg-element-entropy}.
\end{proof}


\subsection{Global entropy stability}
\label{sec:dg-interface-entropy}

For the face orientation and traces in
\eqref{eq:dg-interface-normal-traces}, entropy stability requires
\begin{equation}
 \jump{V}^T\widehat{\mathbb H}_{n,K,f}^{\LLF}
 -\jump{\Psi_{w,n,K,f}}\le0.
 \label{eq:dg-llf-entropy-inequality}
\end{equation}
\begin{proposition}[Causal LLF stabilization]
\label{thm:pcp-dominates-entropy-speed}
For two admissible states sharing the same prescribed face geometry, the
LLF flux \eqref{eq:dg-classical-llf} with the causal speed
\eqref{eq:dg-causal-normal-speed} satisfies
\eqref{eq:dg-llf-entropy-inequality}.
\end{proposition}
\begin{proof}
See \ref{app:llf-speed-bound}.
\end{proof}
Section~\ref{sec:pcp} shows that the same causal bound also provides the
Lax--Friedrichs splitting used for physical-constraint preservation.

Define
\begin{equation}
 \mathcal E_h(t)
 :=\sum_K\mathcal E_K
 =\sum_KJ_K\sum_{i,j=0}^N
 \omega_i\omega_j\,w_{ij}\eta_{ij}.
 \label{eq:dg-discrete-entropy}
\end{equation}
\begin{theorem}[Semidiscrete entropy stability and weighted rest-mass conservation]
\label{thm:semidiscrete}
Consider \eqref{eq:dgsem-affine} with the LLF interface flux
\eqref{eq:dg-classical-llf} on a conforming affine tensor-product mesh.
Assume that $w>0$ at every volume and face GLL node, all nodal states are
admissible, and the prescribed geometry is single-valued at shared faces.
Under periodic boundary conditions, the scheme satisfies
\begin{equation}
 \frac{d\mathcal E_h}{dt}\le0,
 \label{eq:dg-global-entropy-inequality}
\end{equation}
and conserves the global weighted rest mass.
\end{theorem}
\begin{proof}
Summing \eqref{eq:dg-element-entropy} over the mesh and applying
\eqref{eq:dg-llf-entropy-inequality} on each interior face gives
\eqref{eq:dg-global-entropy-inequality}; periodicity removes the exterior
boundary terms.  Weighted rest-mass conservation follows from
\eqref{eq:dg-element-balance} and cancellation of shared numerical fluxes.
\end{proof}



%% file: sections/05_pcp.tex
\section{Physical-constraint-preserving analysis}\label{sec:pcp}

We apply the transformed-state PCP framework of Wu~\cite{Wu2017} to the
present positive-weight DGSEM.  The transformed cone, Lax--Friedrichs
splitting, convex-decomposition argument, and conservative scaling limiter are
all independent of the spatial metric.  The reduction weight is assumed to
satisfy $w>0$ at the nodes, but it may vary spatially.

\subsection{Metric-independent admissible cone}

For a state vector $Z=(D,m,E)^T$, set
\begin{equation}
 \begin{aligned}
  R(Z)&:=\sqrt{D^2+|m|^2},
  &q(Z)&:=E-R(Z),\\
  \Gset&:=\{Z:D>0,\ q(Z)>0\}.
 \end{aligned}
 \label{eq:cone}
\end{equation}
For the Gamma-law gas, $\Gset$ is equivalent to positivity of the physical
rest-mass density and pressure together with the subluminal-velocity
constraint; see Wu~\cite{Wu2017}.  Its closure $\overline{\Gset}$ is obtained
by replacing the strict inequalities by non-strict ones.

\begin{lemma}[Cone properties]\label{lem:pcp-cone}
The set $\Gset$ is an open convex cone and $\overline{\Gset}$ is a closed
convex cone.  Moreover, for every $c>0$,
\[
 Z\in\Gset\iff cZ\in\Gset,
\]
and if $X\in\Gset$ and $Y_0\in\overline{\Gset}$, then
$X+Y_0\in\Gset$.
\end{lemma}
\begin{proof}
The Euclidean norm is convex and positively homogeneous, hence $q$ is
concave and positively homogeneous.  Convexity and positive scaling follow
immediately.  The last assertion follows from
\[
 q(X+Y_0)\ge q(X)+q(Y_0)>0,
\]
together with positivity of the density component.
\end{proof}

At every node,
\begin{equation}
 \mathbb W=wW=wj_2\uhat,
 \qquad wj_2>0,
 \label{eq:pcp-weighted-state}
\end{equation}
so positive scaling gives the equivalent admissibility tests
\begin{equation}
 \uhat\in\Gset
 \iff W=j_2\uhat\in\Gset
 \iff \mathbb W=wW\in\Gset.
 \label{eq:pcp-state-equivalence}
\end{equation}
Thus the local orthonormal state, intrinsic W-state, and evolved state all use
the same metric-independent cone.

\subsection{Coordinate-direction Lax--Friedrichs splitting and source control}

For each fixed $I\in\{1,2\}$, with no summation over $I$, the coordinate flux
is
\begin{equation}
 \mathbb H^I(\mathbb W)
 =\mathcal A_{aI}\fhat^a(\uhat)-\mathcal C^I\uhat,
 \qquad
 \uhat=\frac{\mathbb W}{wj_2}.
 \label{eq:pcp-coordinate-flux}
\end{equation}
The intrinsic geometry map \eqref{eq:dg-geometry-inverse-map} gives
\[
 \frac{C^I}{j_2}=\beta^I,
 \qquad
 \frac{\sqrt{A_{aI}A_{aI}}}{j_2}
 =\alpha\sqrt{\gamma^{II}}.
\]
Accordingly, define the coordinate-direction causal bound
\begin{equation}
 a_I^{\mathrm{PCP}}
 :=|\beta^I|+\alpha\sqrt{\gamma^{II}}
 =\frac{|C^I|+\sqrt{A_{aI}A_{aI}}}{j_2},
 \qquad I=1,2.
 \label{eq:pcp-speed}
\end{equation}

\begin{lemma}[Lax--Friedrichs splitting]\label{lem:pcp-lxf-split}
Let $w>0$ and $\mathbb W\in\Gset$.  For either coordinate direction $I$, if
$a\ge a_I^{\mathrm{PCP}}$, then
\begin{equation}
 \mathbb W\pm\frac{\mathbb H^I(\mathbb W)}{a}
 \in\overline{\Gset}.
 \label{eq:pcp-lxf-splitting}
\end{equation}
\end{lemma}
\begin{proof}
By rotational invariance and the Lax--Friedrichs splitting property of the
SRHD admissible set \cite[Lemma~2.3(ii)--(iii)]{WuTang2015}, together with
the subluminal SRHD characteristic speeds, for every unit orthonormal vector
$\widehat e$,
\[
 \uhat\pm c\,\widehat e_a\fhat^a(\uhat)
 \in\overline{\Gset},
 \qquad 0\le c\le1.
\]
For fixed $I$, the vector
$\mathcal A_{aI}/\sqrt{\mathcal A_{bI}\mathcal A_{bI}}$ is a unit
orthonormal direction.  Using \eqref{eq:pcp-coordinate-flux},
\begin{align}
 \mathbb W\pm\frac{\mathbb H^I}{a}
 ={}&wj_2\left(1\mp\frac{\mathcal C^I}{awj_2}\right)
 \left[
 \uhat\pm
 \frac{\mathcal A_{aI}}
 {awj_2\mp\mathcal C^I}\fhat^a
 \right].
 \label{eq:pcp-lxf-factorization}
\end{align}
If $a\ge a_I^{\mathrm{PCP}}$, then
\[
 awj_2\mp\mathcal C^I
 \ge awj_2-|\mathcal C^I|
 \ge\sqrt{\mathcal A_{aI}\mathcal A_{aI}}.
\]
The exterior factor is therefore nonnegative and the vector coefficient of
$\fhat^a$ has norm at most one.  The local SRHD splitting and cone-scaling
property complete the proof.
\end{proof}

At a volume node, define the source escape rate by setting
$\lambda_S=0$ if $\mathbb S^h\in\overline{\Gset}$; otherwise let
$\lambda_S>0$ be determined by the first intersection of the source ray with
$\partial\Gset$,
\[
 \mathbb W+t\mathbb S^h\in\Gset
 \quad\text{for }0\le t<\lambda_S^{-1},
 \qquad
 \mathbb W+\frac{\mathbb S^h}{\lambda_S}\in\partial\Gset.
\]
Consequently, whenever $0<\Delta t\lambda_S<1$,
\begin{equation}
 \mathbb W+\Delta t\mathbb S^h
 =(1-\Delta t\lambda_S)\mathbb W
 +\Delta t\lambda_S
 \left(\mathbb W+\frac{\mathbb S^h}{\lambda_S}\right)
 \in\Gset.
 \label{eq:pcp-source-convex-step}
\end{equation}

\subsection{Forward-Euler cell-average PCP property}
\label{sec:pcp-fe-positive-weight}

Applying the standard high-order convex-decomposition argument
\cite{QinShuYang2016,Wu2017} gives a sufficient cell-average condition.  Let
$K$ be an affine tensor-product element as in
Section~\ref{sec:dg-affine-elements}.  Its
quadrature cell average is
\begin{equation}
 \overline{\mathbb W}_K
 :=\frac14\sum_{\boldsymbol\ell}
 \omega_{\ell_1}\omega_{\ell_2}
 \mathbb W_{\boldsymbol\ell},
 \qquad
 \boldsymbol\ell=(\ell_1,\ell_2).
 \label{eq:pcp-gll-average}
\end{equation}
By Proposition~\ref{lem:dg-single-element-balance}, one forward-Euler update
is
\begin{align}
 \overline{\mathbb W}_K^{\,n+1}
 ={}&\frac14\sum_{\boldsymbol\ell}
 \omega_{\ell_1}\omega_{\ell_2}
 \left(\mathbb W_{\boldsymbol\ell}
 +\Delta t\,\mathbb S^h_{\boldsymbol\ell}\right)\notag\\
 &-\frac{\Delta t}{|K|}
 \sum_{f\subset\partial K}\sum_{q=0}^N
 J_f\omega_q\widehat{\mathbb H}_{n,f,q}^{\LLF}.
 \label{eq:average-update}
\end{align}

\begin{theorem}[Forward-Euler cell-average PCP]\label{thm:fe-pcp}
Assume $w>0$ at every volume GLL node of $K$ and at every face node at which
an LLF flux is evaluated.  Assume that all volume states and both states at
every such face node belong to $\Gset$, that the prescribed face geometry is
single-valued, and that
\[
 a_{f,q}\ge a_{I(f)}^{\mathrm{PCP}}.
\]
Let $\boldsymbol\ell(f,q)$ denote the volume GLL node of $K$ coinciding with
the face node $(f,q)$.  If, for every volume node,
\begin{equation}
 \Delta t\left[
 \lambda_{S,\boldsymbol\ell}
 +\frac{4}
 {\omega_{\ell_1}\omega_{\ell_2}|K|}
 \sum_{\substack{(f,q):\\
 \boldsymbol\ell(f,q)=\boldsymbol\ell}}
 \omega_qJ_fa_{f,q}
 \right]<1,
 \label{eq:nodal-cfl}
\end{equation}
then $\overline{\mathbb W}_K^{\,n+1}\in\Gset$.
\end{theorem}
\begin{proof}
Order the two traces at each LLF face so that $\mathbb W_{f,q}^-$ is the
trace from $K$.  Expanding the LLF flux in \eqref{eq:average-update} groups
each face contribution into nonnegative multiples of
\[
 \mathbb W_{f,q}^{\pm}
 -\frac{\varepsilon_{K,f}}{a_{f,q}}
 \mathbb H^{I(f)}(\mathbb W_{f,q}^{\pm})
 \in\overline{\Gset},
\]
by Lemma~\ref{lem:pcp-lxf-split}, while subtracting
$\Delta t\,\omega_qJ_fa_{f,q}/|K|$ from the coefficient of the
corresponding interior nodal state.  At each volume node, the source is treated
by \eqref{eq:pcp-source-convex-step}.  Condition~\eqref{eq:nodal-cfl} is
precisely the positivity condition for the remaining coefficient of every
strictly admissible volume state.  Convexity of the cone then proves the
claim.
\end{proof}

\subsection{Conservative weight-compatible scaling}
\label{sec:pcp-weighted-scaling}

We use a single weight-compatible form of the conservative scaling framework
of Zhang and Shu~\cite{ZhangShu2010}, with its relativistic adaptations
\cite{QinShuYang2016,Wu2017}, for the densitized polynomial
$\mathbb W_h=\chi\widehat U_h$, where
\begin{equation}
 \chi_{\boldsymbol\ell}:=w_{\boldsymbol\ell}j_{2,\boldsymbol\ell}>0.
 \label{eq:pcp-chi-definition}
\end{equation}
 Let $\mu_{\boldsymbol\ell}>0$ be the mapped GLL mass weights and
define
\begin{equation}
 \overline{\mathbb W}_K
 :=\frac{\sum_{\boldsymbol\ell}\mu_{\boldsymbol\ell}
 \mathbb W_{\boldsymbol\ell}}
 {\sum_{\boldsymbol\ell}\mu_{\boldsymbol\ell}},
 \qquad
 \overline\chi_K
 :=\frac{\sum_{\boldsymbol\ell}\mu_{\boldsymbol\ell}
 \chi_{\boldsymbol\ell}}
 {\sum_{\boldsymbol\ell}\mu_{\boldsymbol\ell}}.
 \label{eq:pcp-weighted-averages}
\end{equation}
Assume $\overline\chi_K>0$ and
$\overline{\mathbb W}_K\in\Gset$.  Cone homogeneity then gives the local
orthonormal anchor
\begin{equation}
 \widehat U_{A,K}
 :=\frac{\overline{\mathbb W}_K}{\overline\chi_K}\in\Gset,
 \qquad
 \mathbb W_{\boldsymbol\ell}^{\rm ref}
 :=\chi_{\boldsymbol\ell}\widehat U_{A,K}.
 \label{eq:pcp-weight-anchor}
\end{equation}
The reference polynomial represents a constant local orthonormal state and
has exactly the original conservative average.  Write
\begin{equation}
 \widehat U_{\boldsymbol\ell}^{(0)}
 :=\frac{\mathbb W_{\boldsymbol\ell}}{\chi_{\boldsymbol\ell}}.
 \label{eq:pcp-initial-local-state}
\end{equation}

For $Z=(D,m,E)^T\in\Gset$, define the normalized cone margin
\begin{equation}
 \nu(Z):=\frac{q(Z)}{E+R(Z)}
 =\frac{E-R(Z)}{E+R(Z)}\in(0,1],
 \label{eq:pcp-normalized-margin}
\end{equation}
and, for $0\le\delta<1$, set
\begin{equation}
 g_\delta(Z):=(1-\delta)E-(1+\delta)R(Z).
 \label{eq:pcp-relative-margin}
\end{equation}
The condition $g_\delta(Z)\ge0$ is equivalent to $\nu(Z)\ge\delta$.

To match the numerical floors without demanding more than the admissible
anchor itself contains, define
\begin{equation}
 \varepsilon_{D,K}^{\rm eff}
 :=\min\{\varepsilon_D,(\widehat U_{A,K})_D\},
 \qquad
 \varepsilon_{q,K}^{\rm eff}
 :=\min\{\varepsilon_q,q(\widehat U_{A,K})\},
 \label{eq:pcp-effective-floors}
\end{equation}
and set
\begin{equation}
 \delta_K:=\eta_{\rm PCP}\,\nu(\widehat U_{A,K}),
 \qquad 0\le\eta_{\rm PCP}<1.
 \label{eq:pcp-weighted-margin-target}
\end{equation}
In the numerical implementation, we take
$\varepsilon_D=\varepsilon_q=10^{-11}$ and $\eta_{\rm PCP}=0.1$.
The corresponding closed convex certification set is
\begin{equation}
 \Gset_K^{\rm cert}
 :=\left\{
 Z:\ D\ge\varepsilon_{D,K}^{\rm eff},\quad
 q(Z)\ge\varepsilon_{q,K}^{\rm eff},\quad
 g_{\delta_K}(Z)\ge0
 \right\}.
 \label{eq:pcp-certification-set}
\end{equation}
It contains $\widehat U_{A,K}$, with strict relative-margin slack when
$\eta_{\rm PCP}<1$.

The limiter applies three explicit conservative scaling steps.  First set
\begin{equation}
 D_{\min}^{(0)}:=\min_{\boldsymbol\ell}
 (\widehat U_{\boldsymbol\ell}^{(0)})_D,
 \qquad
 \theta_D:=
 \begin{cases}
  1,
  &D_{\min}^{(0)}\ge\varepsilon_{D,K}^{\rm eff},\\[2mm]
  \displaystyle
  \frac{(\widehat U_{A,K})_D-\varepsilon_{D,K}^{\rm eff}}
       {(\widehat U_{A,K})_D-D_{\min}^{(0)}},
  &D_{\min}^{(0)}<\varepsilon_{D,K}^{\rm eff}.
 \end{cases}
 \label{eq:pcp-density-factor}
\end{equation}
Only the density component is scaled in this step:
\begin{equation}
 \begin{aligned}
 (\mathbb W_{\boldsymbol\ell}^{(1)})_D
 &= (\mathbb W_{\boldsymbol\ell}^{\rm ref})_D
 +\theta_D\left[(\mathbb W_{\boldsymbol\ell})_D
 -(\mathbb W_{\boldsymbol\ell}^{\rm ref})_D\right],\\
 (\mathbb W_{\boldsymbol\ell}^{(1)})_{m_a}
 &= (\mathbb W_{\boldsymbol\ell})_{m_a},\quad a=1,2,
 \qquad
 (\mathbb W_{\boldsymbol\ell}^{(1)})_E
 = (\mathbb W_{\boldsymbol\ell})_E.
 \end{aligned}
 \label{eq:pcp-density-stage}
\end{equation}
Denote $\widehat U_{\boldsymbol\ell}^{(1)}
:=\frac{\mathbb W_{\boldsymbol\ell}^{(1)}}
{\chi_{\boldsymbol\ell}}$.

Next define
\begin{equation}
 q_{\min}^{(1)}:=\min_{\boldsymbol\ell}
 q(\widehat U_{\boldsymbol\ell}^{(1)}),
 \qquad
 \theta_q:=
 \begin{cases}
  1,
  &q_{\min}^{(1)}\ge\varepsilon_{q,K}^{\rm eff},\\[2mm]
  \displaystyle
  \frac{q(\widehat U_{A,K})-\varepsilon_{q,K}^{\rm eff}}
       {q(\widehat U_{A,K})-q_{\min}^{(1)}},
  &q_{\min}^{(1)}<\varepsilon_{q,K}^{\rm eff},
 \end{cases}
 \label{eq:pcp-cone-factor}
\end{equation}
and scale the complete state with this single elementwise factor:
\begin{equation}
 \mathbb W_{\boldsymbol\ell}^{(2)}
 :=\mathbb W_{\boldsymbol\ell}^{\rm ref}
 +\theta_q\left(
 \mathbb W_{\boldsymbol\ell}^{(1)}
 -\mathbb W_{\boldsymbol\ell}^{\rm ref}\right),
 \qquad
 \widehat U_{\boldsymbol\ell}^{(2)}
 :=\frac{\mathbb W_{\boldsymbol\ell}^{(2)}}
 {\chi_{\boldsymbol\ell}}.
 \label{eq:pcp-cone-stage}
\end{equation}

Finally, with
\begin{equation}
 g_{A,K}:=g_{\delta_K}(\widehat U_{A,K})
 =(1-\eta_{\rm PCP})q(\widehat U_{A,K})>0,
 \qquad
 g_{\min}^{(2)}:=\min_{\boldsymbol\ell}
 g_{\delta_K}(\widehat U_{\boldsymbol\ell}^{(2)}),
 \label{eq:pcp-margin-data}
\end{equation}
set
\begin{equation}
 \theta_\delta:=
 \begin{cases}
  1,&g_{\min}^{(2)}\ge0,\\[1mm]
  \displaystyle\frac{g_{A,K}}{g_{A,K}-g_{\min}^{(2)}},
  &g_{\min}^{(2)}<0,
 \end{cases}
 \qquad
 \mathbb W_{\boldsymbol\ell}^{L}
 :=\mathbb W_{\boldsymbol\ell}^{\rm ref}
 +\theta_\delta\left(
 \mathbb W_{\boldsymbol\ell}^{(2)}
 -\mathbb W_{\boldsymbol\ell}^{\rm ref}\right).
 \label{eq:pcp-margin-stage}
\end{equation}
The superscripts $(1)$, $(2)$, and $L$ denote the outputs of the density,
cone, and final margin stages, respectively, and
$\widehat U_{\boldsymbol\ell}^{L}
:=\mathbb W_{\boldsymbol\ell}^{L}/\chi_{\boldsymbol\ell}$ is the final local
orthonormal state at GLL node $\boldsymbol\ell$.  The first two stages impose
absolute floors on $D$ and $q$.  An absolute $q$ floor, however, does not
control the relative proximity to the cone boundary $E=R$: a large-magnitude
state may satisfy that floor while $\nu=q/(E+R)$ is arbitrarily small.  The
third stage therefore provides a scale-invariant numerical buffer.  By
construction,
$\nu(\widehat U_{\boldsymbol\ell}^{L})
\ge\eta_{\rm PCP}\,\nu(\widehat U_{A,K})$, so every nodal state retains a fixed
fraction of the anchor state's normalized cone margin.  This improves the
robustness of primitive recovery near the admissible-set boundary.

\begin{theorem}[Conservative weight-compatible PCP scaling]
\label{thm:pcp-weighted-scaling}
Assume $\chi_{\boldsymbol\ell}>0$ at every node,
$\overline\chi_K>0$, and $\overline{\mathbb W}_K\in\Gset$.  The explicit
three-stage scaling \eqref{eq:pcp-density-stage},
\eqref{eq:pcp-cone-stage}, and \eqref{eq:pcp-margin-stage} preserves the
conservative cell average and places every nodal $\widehat U$ state in
$\Gset_K^{\rm cert}$.  If the unlimited polynomial already satisfies all
certification conditions, then $\theta_D=\theta_q=\theta_\delta=1$.
\end{theorem}
\begin{proof}
The anchor belongs to $\Gset$ by the positive-scaling property of
Lemma~\ref{lem:pcp-cone}.  Because $(\widehat U_{A,K})_D$ is the
$\mu_{\boldsymbol\ell}\chi_{\boldsymbol\ell}$-weighted mean of the nodal
local densities, \eqref{eq:pcp-density-factor}--
\eqref{eq:pcp-density-stage} impose the density floor while preserving the
density component of $\overline{\mathbb W}_K$; the remaining components are
unchanged.  Concavity of $q$ gives
\[
 q(\widehat U_{\boldsymbol\ell}^{(2)})
 \ge q(\widehat U_{A,K})
 +\theta_q\left[q(\widehat U_{\boldsymbol\ell}^{(1)})
 -q(\widehat U_{A,K})\right]
 \ge\varepsilon_{q,K}^{\rm eff}.
\]
The same argument for the concave function $g_{\delta_K}$ shows that
\eqref{eq:pcp-margin-stage} imposes $g_{\delta_K}\ge0$.  The two full-state
contractions preserve the previously imposed constraints because their
superlevel sets are convex and contain the anchor.

Every stage is either componentwise unchanged or a contraction toward
$\mathbb W^{\rm ref}$, whose average equals
$\overline\chi_K\widehat U_{A,K}=\overline{\mathbb W}_K$.
Consequently all stages preserve the conservative average.  If the unlimited
states already satisfy the three conditions, each definition selects unity.
\end{proof}

The same scaling has a conservative extension to axis cells on which $w$ has
a simple normal zero, including the polar-axis cells used below; see
\ref{app:simple-zero-treatment}.

\subsection{Stage-average rejection}
\label{sec:pcp-stage-rejection}

The numerical floors are distinct from the strict admissibility of the
conservative cell average. If a raw forward-Euler candidate or an SSP convex
combination satisfies
\begin{equation}
 (\overline{\mathbb W}_K)_D\le0
 \qquad\text{or}\qquad
 q(\overline{\mathbb W}_K)\le0
 \quad\text{for some }K,
 \label{eq:pcp-average-rejection}
\end{equation}
the candidate step is rejected and recomputed with the timestep reduced by a
factor of two. This retry procedure does not drive the timestep indefinitely
toward zero: Theorem~\ref{thm:fe-pcp} guarantees a positive admissible
timestep threshold under which the forward-Euler cell averages remain in
$\Gset$. Hence, after finitely many halvings, the stage-average admissibility
test is satisfied.

If the cell average lies in $\Gset$ but falls below a prescribed numerical
floor, the step is not rejected. Instead, the corresponding scaling may
collapse the polynomial to its cell average; for a variable-densitization
element, the effective requirements \eqref{eq:pcp-effective-floors} likewise
collapse to the admissible local anchor when necessary. Thus the numerical
floors do not alter an already admissible conservative cell average.

%% file: sections/06_algorithm.tex
\section{Fully discrete scheme}\label{sec:algorithm}

We complete the DGSEM specification used in the numerical experiments by
describing primitive recovery, classical LLF interface fluxes, stage-wise
oscillation elimination (OE), conservative PCP scaling, SSPRK$(3,3)$ time
stepping, and the geometry-based explicit timestep.

\subsection{Residual evaluation and causal face stabilization}
\label{sec:algorithm-spatial}

The stationary numerical geometry state is
\[
 Y
 =\left(j_2,A_{11},A_{12},A_{22},C^1,C^2\right)^T,
 \qquad A_{21}=0,
\]
with the physical reduction weight $w$ stored separately. At the volume and
face GLL nodes, we form the intrinsic geometry, $w$, and the complete
coefficients $\mathcal A=wA$ and $\mathcal C=wC$ once and cache
\[
 (\mathsf D_Ij_2)_{\boldsymbol\ell},\quad
 (\mathsf D_Iw)_{\boldsymbol\ell},\quad
 (\mathsf D_I\mathcal A)_{\boldsymbol\ell},\quad
 (\mathsf D_I\mathcal C)_{\boldsymbol\ell},
 \qquad I=1,2.
\]
Whenever ADM quantities or their derivatives are needed, they are reconstructed
from the intrinsic $Y$ and its derivatives through
\eqref{eq:dg-geometry-inverse-map} and
\eqref{eq:dg-directional-alpha}--\eqref{eq:dg-directional-gamma}.

At nodes where $w>0$, the reduction weight permits conversion of the
evolved state to the local orthonormal conservative state by
\begin{equation}
 \widehat U_{\boldsymbol\ell}
 =\frac{\mathbb W_{\boldsymbol\ell}}
 {w_{\boldsymbol\ell}j_{2,\boldsymbol\ell}}
 =\left(D,\widehat m_1,\widehat m_2,E\right)^T.
 \label{eq:algorithm-local-state}
\end{equation}
For an axis cell with a simple zero of $w$ on its axis face, the regular trace
is supplied without dividing by the stored zero state; see
\ref{app:simple-zero-treatment}.

For the Gamma-law equation of state, primitive recovery reduces to a scalar
nonlinear solve for the pressure.  Once $p$ is obtained,
\begin{equation}
 \widehat v_a=\frac{\widehat m_a}{E+p},
 \qquad
 L=(1-|\widehat v|^2)^{-1/2},
 \qquad
 \rho=\frac{D}{L},
 \qquad
 h=1+\frac{\Gamma}{\Gamma-1}\frac{p}{\rho}.
 \label{eq:algorithm-primitive-reconstruction}
\end{equation}
The recovered primitive variables determine the four-velocity and the
stress--energy tensor in \eqref{eq:perfect-fluid}, and hence the Valencia
source \eqref{eq:valencia-source}. Metric derivatives are reconstructed from
the cached intrinsic geometry derivatives, and the weighted source is
assembled directly from the finite product derivatives of
Section~\ref{sec:dgsem},
\begin{align}
 \mathbb S^h_{\boldsymbol\ell}
 ={}&\sum_{I=1}^2\kappa_I\bigl[
 wB_{j_2}^I\mathsf D_Ij_2
 +B_{A_{aJ}}^I\mathsf D_I(wA_{aJ})
 +B_{C^J}^I\mathsf D_I(wC^J)
 +R_w^I\mathsf D_Iw\bigr]_{\boldsymbol\ell}.
 \label{eq:algorithm-source}
\end{align}

The volume term uses the entropy-conservative flux of
Section~\ref{sec:dgsem}, while element interfaces use the classical LLF flux
\eqref{eq:dg-classical-llf} with the dissipation coefficient $a_f$ in
\eqref{eq:dg-causal-normal-speed}. This geometry-only coefficient supplies the
dissipation required by both the classical-LLF entropy inequality and the PCP
Lax--Friedrichs splitting. We denote the resulting semidiscrete operator by
\[
 \frac{d\mathbb W_h}{dt}=\mathcal R_h(\mathbb W_h).
\]

\subsection{Conservative local-state oscillation elimination}
\label{sec:algorithm-oe}

High-order DG approximations may develop spurious oscillations near shocks and
other under-resolved structures. As the first stabilization step at each
Runge--Kutta stage, we apply an OE procedure similar to those in
\cite{PengSunWu2025,CaoPengWu2025} to the regular orthonormal-frame
conservative state $\widehat U$, rather than to the densitized variable
$\mathbb W=\chi\widehat U$. This choice prevents the stationary geometric
factor $\chi$ from being interpreted as fluid variation. 
Another key modification is the direction-matched cross-line normalization used in the OE indicator, introduced below.

\subsubsection{Directional OE indicator}

On an affine tensor-product element $K$, let
$\widehat U_h\in\mathbb Q_N(K)^4$ be the GLL interpolant of the nodal states
in~\eqref{eq:algorithm-local-state}, let $\mu_{K,\boldsymbol\ell}>0$ be the
mapped GLL mass weights, and let $\Delta x_{K,I}$ be the element width in direction
$I$. For $m=0,\ldots,N$, the order-$m$ damping rate is
\begin{equation}
\delta_m^{(K)}
:=\sum_{I=1}^2
\lambda_{K,I}^{\rm PCP}\,
\sigma_{m,I}^{(K)}(\widehat U_h),
\label{eq:oe-damping-rate}
\end{equation}
where the effective directional causal rate is
\begin{equation}
\lambda_{K,I}^{\rm PCP}
:=\max_{\boldsymbol\ell\in K}
\frac{a_I^{\rm PCP}(x_{\boldsymbol\ell})}
{\Delta x_{K,I}/(2N+1)},
\qquad
\lambda_K^{\rm PCP}
:=\max_{I=1,2}\lambda_{K,I}^{\rm PCP}.
\label{eq:oe-lambda}
\end{equation}
Here $a_I^{\rm PCP}$ is the geometry-only causal speed
\eqref{eq:pcp-speed}, and $\Delta x_{K,I}/(2N+1)$ is the effective DG resolution
length.

To construct the dimensionless directional OE indicator
$\sigma_{m,I}^{(K)}$, let
$\mathcal F_{K,I}^{\rm int}$ denote the set of interior faces of $K$ normal
to direction $I$, and define the corresponding jump amplitude by
\begin{subequations}
\label{eq:oe-directional-indicators}
\begin{equation}
\mathcal J_{m,I,K}^{(k)}
:=
\sum_{f\in\mathcal F_{K,I}^{\rm int}}
\left[
\frac1{|f|}\int_f
\left|\jump{\partial_I^m\widehat U_h^{(k)}}\right|^2\,\dd s
\right]^{1/2},
\label{eq:oe-directional-jump}
\end{equation}
where the face integrals are evaluated by one-dimensional GLL quadrature.
For every component--direction pair retained by the noise-level screening
specified below, the componentwise directional OE indicator is
\begin{equation}
\sigma_{m,I}^{(K)}(\widehat U_h^{(k)})
:=
\frac{(2m+1)(\Delta x_{K,I})^m}
{2\,m!\,\Delta_{k,I,K}}
\mathcal J_{m,I,K}^{(k)}.
\label{eq:oe-directional-indicator}
\end{equation}
\end{subequations}
 The
system directional OE indicator is obtained by taking the component maximum
separately in each direction:
\begin{equation}
\sigma_{m,I}^{(K)}(\widehat U_h)
:=\max_{1\le k\le4}
\sigma_{m,I}^{(K)}(\widehat U_h^{(k)}).
\label{eq:oe-system-indicator}
\end{equation}
The factor $(\Delta x_{K,I})^m$ balances the $m$th derivative jump, while
$\Delta_{k,I,K}$ supplies the component amplitude against which that jump is
measured.

\subsubsection{Cross-line normalization}
\label{sec:oe-cross-line-scaling}

The choice of this amplitude requires particular care.  A domain-wide
normalization is
\begin{equation}
\overline{\widehat U}^{(k)}_{\Omega}
 :=\frac{
\sum_{K'}\sum_{\boldsymbol\ell\in K'}
\mu_{K',\boldsymbol\ell}\widehat U_{K',\boldsymbol\ell}^{(k)}}
 {
\sum_{K'}\sum_{\boldsymbol\ell\in K'}
\mu_{K',\boldsymbol\ell}},
\qquad
\Delta_k^{\Omega}
:=\max_{K',\,\boldsymbol\ell\in K'}
\left|
\widehat U_{K',\boldsymbol\ell}^{(k)}
-\overline{\widehat U}^{(k)}_{\Omega}
\right|.
\label{eq:oe-global-reference-amplitude}
\end{equation}
Although $\Delta_k^{\Omega}$ makes the indicator dimensionless, it couples every
element to the largest excursion anywhere in the domain.  This is problematic
in black-hole accretion, where conservative components may vary by several
orders of magnitude between the near-horizon region and the outer flow.  A
large near-hole excursion can therefore make normalized jumps elsewhere
artificially small.

To localize the normalization without using a purely elementwise amplitude,
we introduce a cross-line scale on the structured tensor-product meshes used
here.  For $I\in\{1,2\}$, let $\mathcal X_I(K)$ be the coordinate line through
$K$ that varies in direction $I$ and is fixed in the transverse direction.
For fixed $I$ and $K$, abbreviate this line by
$\mathcal X=\mathcal X_I(K)$, and let
$\overline{\widehat U}^{(k)}_{\mathcal X}$ denote the corresponding weighted
line mean, obtained by restricting both sums in the first definition of
\eqref{eq:oe-global-reference-amplitude} to $K'\in\mathcal X$.  The raw line
amplitude and the denominator used in~\eqref{eq:oe-directional-indicator} are
\begin{equation}
\widetilde\Delta_{k,I,K}
:=\max_{\substack{K'\in\mathcal X\\
\boldsymbol\ell\in K'}}
\left|
\widehat U_{K',\boldsymbol\ell}^{(k)}
-\overline{\widehat U}^{(k)}_{\mathcal X}
\right|,
\qquad
\Delta_{k,I,K}
:=\max\left\{
\widetilde\Delta_{k,I,K},10^{-6}\Delta_k^{\Omega}
\right\}.
\label{eq:oe-cross-line-normalization}
\end{equation}
The implementation screens noise-level amplitudes before evaluating
\eqref{eq:oe-directional-indicator}.  Set
$U_{\max}:=\max_{K',\,\boldsymbol\ell\in K'}
\|\widehat U_{K',\boldsymbol\ell}\|_2$,
$\tau_{\rm cmp}:=\max\{10^{-8}U_{\max},
10^3\epsilon_{\rm mach}\}$, and
$\tau_{\rm abs}:=10^3\epsilon_{\rm mach}$, where
$\epsilon_{\rm mach}$ is machine precision.  Component $k$ contributes only
when $\Delta_k^{\Omega}>\tau_{\rm cmp}$; among retained components, a
component--direction pair is omitted whenever
$\Delta_{k,I,K}\le\tau_{\rm abs}$.  These safeguards prevent division by a
noise-level amplitude: the global test screens numerically constant
components, while the floor in~\eqref{eq:oe-cross-line-normalization} protects
a nearly constant coordinate line.  The element-attached cross-line scale
$\Delta_{k,K}:=(\Delta_{k,1,K},\Delta_{k,2,K})$ retains separate
normalizations for the two coordinate directions.  In Cartesian coordinates,
the $x$-normal jump uses the fixed-$y$ line through $K$, whereas the
$y$-normal jump uses the fixed-$x$ line.  In $(r,\theta)$ or
$(r,\widetilde\phi)$ coordinates, these become fixed-angle radial lines and
constant-$r$ angular lines.  A large radial variation therefore does not mask
an angular discontinuity, or vice versa, and the construction is symmetric
under coordinate interchange. Figure~\ref{fig:oe-cross-line-scaling} contrasts the domain-wide and
direction-matched normalizations.
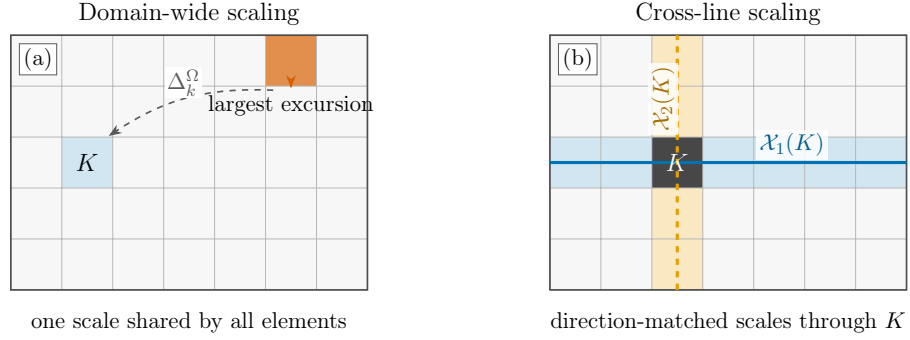
\begin{figure}[t]
  \centering
  \resizebox{0.88\linewidth}{!}{%
    \input{figures/oe_cross_line_scaling.tex}%
  }
  \caption{Schematic of the normalization used by the directional OE
  indicator.  (a) A domain-wide scale $\Delta_k^{\Omega}$ is shared by every
  element, so a large excursion far from $K$ can control its normalization.
  (b) Cross-line scaling assigns to $K$ a direction-matched pair:
  $\Delta_{k,1,K}$ is computed along the fixed-$y$ line $\mathcal X_1(K)$,
  whereas $\Delta_{k,2,K}$ is computed along the fixed-$x$ line
  $\mathcal X_2(K)$.}
  \label{fig:oe-cross-line-scaling}
\end{figure}

With this choice, the indicator is invariant under a componentwise affine
rescaling of $\widehat U$.  Hence $\sigma_{m,I}^{(K)}$ is dimensionless and
$\delta_m^{(K)}$ has units of inverse time.  We refer to this construction as
cross-line scaling and use it in all OE calculations reported in
Section~\ref{sec:numerics}.

\subsubsection{Conservative OE map}

Because $\mathbb W_h=\chi\widehat U_h$, preserving the ordinary mean of
$\widehat U_h$ is insufficient when $\chi$ varies on $K$. We therefore
introduce the $\chi$-weighted mean
\begin{equation}
\langle Z\rangle_{\chi,K}
:=\frac{\displaystyle
\sum_{\boldsymbol\ell}\mu_{K,\boldsymbol\ell}
\chi_{\boldsymbol\ell}Z_{\boldsymbol\ell}}
{\displaystyle
\sum_{\boldsymbol\ell}\mu_{K,\boldsymbol\ell}
\chi_{\boldsymbol\ell}},
\qquad
\widehat U_{A,K}
:=\langle\widehat U_h\rangle_{\chi,K}.
\label{eq:oe-weighted-anchor}
\end{equation}
The resulting $\widehat U_{A,K}$ is simply the PCP
anchor~\eqref{eq:pcp-weight-anchor}. Let
$P_{\chi,K}^r:\mathbb Q_N(K)\to\mathbb Q_r(K)$ denote the componentwise
$\chi$-weighted GLL projection. Its hierarchical increments are defined by
\begin{equation}
\begin{aligned}
\sum_{\boldsymbol\ell}\mu_{K,\boldsymbol\ell}\chi_{\boldsymbol\ell}
\left[(P_{\chi,K}^r Z)_{\boldsymbol\ell}-Z_{\boldsymbol\ell}\right]
V_{\boldsymbol\ell}
&=0
\quad\forall V\in\mathbb Q_r(K),\\
\widetilde{\mathcal S}_K^r
&:=P_{\chi,K}^r-P_{\chi,K}^{r-1},
\qquad r=1,\ldots,N.
\end{aligned}
\label{eq:oe-weighted-shell}
\end{equation}
Because $\mathbb Q_{r-1}(K)\subset\mathbb Q_r(K)$ and contains the constants,
each increment has zero $\chi$-weighted mean, and
\begin{equation}
\widehat U_h
=\widehat U_{A,K}
+\sum_{r=1}^N
\widetilde{\mathcal S}_K^r\widehat U_h.
\label{eq:oe-weighted-shell-decomposition}
\end{equation}

The degree-$r$ increment is retained by the damping factor
\begin{equation}
\alpha_{r,K}
:=\exp\left[
-s_{\mathrm{OE}}\Delta t\sum_{m=0}^{r}\delta_m^{(K)}
\right],
\qquad r=1,\ldots,N,
\label{eq:oe-shell-retention}
\end{equation}
where the dimensionless parameter $s_{\mathrm{OE}}\ge0$ controls the overall
damping strength. Setting $s_{\mathrm{OE}}=0$ gives the identity map, while
larger values strengthen the damping at fixed indicator and time step.  Its
choice is therefore central to the practical performance of the OE procedure.
We expect its calibration to depend on the polynomial degree.  All numerical
experiments in this work use $\mathbb Q_2$ elements; for this degree, values in
the range $0.01\le s_{\mathrm{OE}}\le0.05$ give robust performance across the
full test suite.  Unless stated otherwise, we use $s_{\mathrm{OE}}=0.02$
throughout.  The accuracy and shock--bubble studies in
Section~\ref{sec:numerics} examine the sensitivity to $s_{\mathrm{OE}}$ and
support $s_{\mathrm{OE}}=0.02$ as a reasonable default for the $\mathbb Q_2$
calculations considered here. The elementwise conservative
OE map is
\begin{equation}
(\mathcal O_K^{\rm cons}\mathbb W_h)_{\boldsymbol\ell}
:=\chi_{\boldsymbol\ell}
\left[
\widehat U_{A,K}
+\sum_{r=1}^N
\alpha_{r,K}
\left(
\widetilde{\mathcal S}_K^r\widehat U_h
\right)_{\boldsymbol\ell}
\right].
\label{eq:oe-conservative-mapback}
\end{equation}
The zero-mean property of every damped increment yields
\begin{equation}
\sum_{\boldsymbol\ell}\mu_{K,\boldsymbol\ell}
(\mathcal O_K^{\rm cons}\mathbb W_h)_{\boldsymbol\ell}
=
\sum_{\boldsymbol\ell}\mu_{K,\boldsymbol\ell}
\mathbb W_{\boldsymbol\ell}.
\label{eq:oe-weighted-conservation}
\end{equation}
Thus the stored conservative cell average is preserved exactly. The treatment
of axis cells, for which $\chi$ vanishes on the axis face, is described
in~\ref{app:simple-zero-treatment}. Applying
$\mathcal O_K^{\rm cons}$ elementwise defines the global conservative OE
operator $\mathcal O_h^{\rm cons}$.


\subsection{PCP stage stabilization}
\label{sec:algorithm-pcp-stage}

After conservative OE, we apply the three-step weight-compatible scaling of
Section~\ref{sec:pcp-weighted-scaling} to the same regular local orthonormal
state. With the anchor \eqref{eq:pcp-weight-anchor}, it enforces the density,
cone, and relative cone-margin constraints while preserving the conservative
element average. Denoting this limiter by $\mathcal L_h^{\rm PCP}$, the
complete stage stabilization is
\begin{equation}
\boxed{
\mathcal P_h
=\mathcal L_h^{\rm PCP}\circ\mathcal O_h^{\rm cons}:
\quad
\text{conservative OE}
\;\longrightarrow\;
\text{PCP scaling}.}
\label{eq:stage-order}
\end{equation}
Here $\mathcal O_h^{\rm cons}$ is the identity when OE is disabled. Both
operators act through $\widehat U$ and preserve the stored conservative cell
average of $\mathbb W$.

\subsection{SSPRK$(3,3)$ stage certification and retry}
\label{sec:algorithm-ssprk}

Let $\mathbb W^n$ be the last accepted solution. The implementation evaluates
SSPRK$(3,3)$ with the OE--PCP stage map $\mathcal P_h$ from
\eqref{eq:stage-order}:
\begin{subequations}
\label{eq:ssprk33}
\begin{align}
 \mathbb W^{(1)}
 &=\mathcal P_h\left(
 \mathbb W^n+\Delta t\,\mathcal R_h(\mathbb W^n)\right),
 \label{eq:ssprk33-stage1}\\
 \mathbb W^{(2)}
 &=\mathcal P_h\left(
 \frac34\mathbb W^n+\frac14\mathbb W^{(1)}
 +\frac{\Delta t}{4}\mathcal R_h(\mathbb W^{(1)})\right),
 \label{eq:ssprk33-stage2}\\
 \mathbb W^{n+1}
 &=\mathcal P_h\left(
 \frac13\mathbb W^n+\frac23\mathbb W^{(2)}
 +\frac{2\Delta t}{3}\mathcal R_h(\mathbb W^{(2)})\right).
 \label{eq:ssprk33-stage3}
\end{align}
\end{subequations}
Before applying $\mathcal P_h$, each raw stage candidate is certified by
requiring the conservative cell average of every element to lie in $\Gset$,
equivalently $(\overline{\mathbb W}_K)_D>0$ and
$q(\overline{\mathbb W}_K)>0$. If certification fails, the whole step is
rejected and retried according to Section~\ref{sec:pcp-stage-rejection}.

\subsection{Causal timestep}
\label{sec:algorithm-time}

The explicit timestep uses the element causal rate $\lambda_K^{\rm PCP}$ from
\eqref{eq:oe-lambda}:
\begin{equation}
 \Delta t
 =\frac{\mathrm{CFL}}
 {\displaystyle\max_K\lambda_K^{\rm PCP}}.
 \label{eq:algorithm-time-step}
\end{equation}
Because the prescribed spacetime and mesh are stationary, the geometry-only
rates $\lambda_K^{\rm PCP}$ are precomputed and reused during time integration.

The sufficient forward-Euler condition \eqref{eq:nodal-cfl} also contains a
source escape rate, which is not included in the nominal hyperbolic timestep
\eqref{eq:algorithm-time-step}. When the source restriction is more severe,
stage-average certification and whole-step retry provide the runtime safeguard
described in Section~\ref{sec:pcp-stage-rejection}.

In summary, each residual evaluation uses the EC volume flux, matched geometry
source, and classical LLF face flux stabilized by the causal $a^{\rm PCP}$.
The same causal speed field supplies the explicit CFL and OE propagation
scale. Each raw SSP stage candidate is certified and then processed by
conservative OE followed by PCP scaling; failed certification
triggers a whole-step retry. This is the fully discrete algorithm used in the
numerical experiments.

%% file: figures/oe_cross_line_scaling.tex
\begin{tikzpicture}[
  x=0.86cm,
  y=0.86cm,
  every node/.style={font=\small},
  panel/.style={draw=black!70, line width=0.8pt},
  mesh/.style={draw=black!28, line width=0.35pt},
  guide/.style={line width=1.5pt},
  >=Stealth
]
  \definecolor{oeBlue}{RGB}{0,114,178}
  \definecolor{oeOrange}{RGB}{230,159,0}
  \definecolor{oeVermillion}{RGB}{213,94,0}

  \begin{scope}
    \fill[black!3] (0,0) rectangle (7,5);
    \fill[oeBlue!18] (1,2) rectangle (2,3);
    \fill[oeVermillion!70] (5,4) rectangle (6,5);
    \draw[mesh,step=1] (0,0) grid (7,5);
    \draw[panel] (0,0) rectangle (7,5);

    \node[draw=black!65, fill=white, inner sep=2pt, anchor=north west]
      at (0.16,4.84) {(a)};
    \node[font=\normalsize] at (3.5,5.42) {Domain-wide scaling};
    \node[font=\normalsize] at (1.5,2.5) {$K$};
    \node[align=center, anchor=south] (excursion) at (5.5,3.25)
      {largest excursion};
    \draw[-{Stealth[length=2.2mm]}, oeVermillion, line width=0.9pt]
      (excursion.north) -- (5.5,3.96);
    \draw[-{Stealth[length=2.2mm]}, dashed, black!65, line width=0.8pt]
      (5.15,3.92) to[bend right=18]
      node[pos=0.52, above, fill=white, inner sep=1pt] {$\Delta_k^{\Omega}$}
      (1.92,3.02);
    \node[anchor=north, align=center] at (3.5,-0.28)
      {one scale shared by all elements};
  \end{scope}

  \begin{scope}[xshift=9.1cm]
    \fill[black!3] (0,0) rectangle (7,5);
    \fill[oeBlue!18] (0,2) rectangle (7,3);
    \fill[oeOrange!24] (2,0) rectangle (3,5);
    \fill[black!72] (2,2) rectangle (3,3);
    \draw[mesh,step=1] (0,0) grid (7,5);
    \draw[panel] (0,0) rectangle (7,5);

    \draw[guide,oeBlue] (0,2.5) -- (7,2.5);
    \draw[guide,oeOrange,dashed] (2.5,0) -- (2.5,5);
    \node[draw=black!65, fill=white, inner sep=2pt, anchor=north west]
      at (0.16,4.84) {(b)};
    \node[font=\normalsize] at (3.5,5.42) {Cross-line scaling};
    \node[text=white,font=\normalsize] at (2.5,2.5) {$K$};

    \node[anchor=south, fill=white, fill opacity=0.82, text opacity=1,
          inner sep=1.5pt, text=oeBlue!80!black] at (4.75,2.56)
      {$\mathcal X_1(K)$};
    \node[rotate=90, anchor=center, fill=white, fill opacity=0.82,
          text opacity=1, inner sep=1.5pt, text=oeOrange!70!black]
      at (2.22,3.75) {$\mathcal X_2(K)$};
    \node[anchor=north, align=center] at (3.5,-0.28)
      {direction-matched scales through $K$};
  \end{scope}
\end{tikzpicture}

%% file: sections/07_numerics.tex
\section{Numerical experiments}
\label{sec:numerics}

We assess the fully discrete scheme of Section~\ref{sec:algorithm} through a sequence of increasingly demanding tests, ranging from smooth special-relativistic flows to multidimensional shocks and black-hole accretion. The experiments are organized to examine, in turn, high-order accuracy in smooth regimes, robustness for strongly nonsmooth relativistic flows, axisymmetric discretizations,
preservation of an exact stationary solution in Schwarzschild spacetime,
and the ability to obtain physically relevant steady accretion states in Schwarzschild and Kerr geometries.

Unless otherwise stated, all calculations use the fully discrete scheme~\eqref{eq:ssprk33} with conforming affine quadrilateral meshes, $\mathbb Q_2$ polynomial approximations, the classical LLF interface flux with the geometry-only causal/PCP dissipation speed, and $\mathrm{CFL}=0.8$. The PCP limiter is enabled in all calculations. For nonsmooth tests, OE damping is applied with $s_{\mathrm{OE}}=0.02$ unless otherwise stated. The two smooth accuracy studies compare the undamped scheme with $s_{\mathrm{OE}}=0.01$, $0.02$, and $0.05$; the shock--bubble study uses the same three nonzero values at fixed resolution.

Table~\ref{tab:numerics-summary} summarizes the principal numerical configurations,
where $w$ denotes the geometric weight entering the weighted conservative formulation.
Case-specific initial and boundary data are given in the corresponding subsections. The implementation is written in C++ using the high-performance MFEM finite element library~\cite{AndersonEtAl2021}.

\begin{table}[t]
     \centering
     \footnotesize
     \caption{Principal numerical configurations used in the numerical experiments.
     PCP is enabled in all calculations; the OE column indicates the damping
     configurations reported.}
     \label{tab:numerics-summary}
     \begin{tabular}{@{}llllll@{}}
     \toprule
     Test & $w$ & Degree and mesh & $\mathrm{CFL}$ & $s_{\mathrm{OE}}$ & Final time \\
     \midrule
     Smooth wave
     & $1$
     & $\mathbb{Q}_2$, $16^2$--$128^2$
     & $0.8$
     & vary
     & $0.1$ \\

     Riemann problems
     & $1$
     & $\mathbb{Q}_2$, $400^2$
     & $0.8$
     & $0.02$
     & $0.4$ \\

     Shock--bubble
     & $1$
     & $\mathbb{Q}_2$, $650\times180$
     & $0.8$
     & vary
     & $450/500$ \\

     C2 jet
     & $r$
     & $\mathbb{Q}_2$, $360\times1080$
     & $0.8$
     & $0.02$
     & $100$ \\

     Michel
     & $r\sin\theta$
     & $\mathbb{Q}_2$, $16\times4$--$128\times4$
     & $0.8$
     & vary
     & $2M$ \\

     Schwarzschild M1a
     & $r\sin\theta$
     & $\mathbb{Q}_2$, $256\times128$
     & $0.8$
     & $0.02$
     & $500M$ \\

     Kerr, four spins
     & $r$
     & $\mathbb{Q}_2$, $200\times160$
     & $0.8$
     & $0.02$
     & $500M$ \\
     \bottomrule
     \end{tabular}
     \end{table}

     \subsection{Two-dimensional smooth-wave accuracy}
     \label{sec:numerics-smooth-wave}

     We first assess the accuracy of the fully discrete method for the planar
     Cartesian SRHD system~\eqref{eq:planar-srhd-balance}. On the periodic unit
     square $\Omega=[0,1]^2$, we consider the smooth traveling wave
     \begin{equation}
     \begin{aligned}
     \rho(x,y,t)
     &=1+0.2\sin\!\left(2\pi\big[(x-v_1t)+(y-v_2t)\big]\right),\\
     (v_1,v_2)
     &=\frac{0.99}{\sqrt{2}}(1,1),
     \qquad p=1,
     \end{aligned}
     \label{eq:numerics-smooth-wave}
     \end{equation}
     with $\Gamma=5/3$. Thus the pressure and velocity remain constant, while the
     density profile is advected with speed $|v|=0.99$.

     For $q=1,2,\infty$, we measure the relative density error
     \begin{equation}
     E_q(\rho)=
     \frac{\|\rho_h-\rho_{\rm ex}\|_{L^q(\dd V)}}
          {\|\rho_{\rm ex}\|_{L^q(\dd V)}},
     \label{eq:numerics-density-errors}
     \end{equation}
     where $\rho_{\rm ex}$ denotes the exact density and
     $\dd V=\dd x\,\dd y$ is the Cartesian volume element.
     We use $\mathbb{Q}_2$ approximations on uniform $16^2$, $32^2$, $64^2$, and
     $128^2$ meshes, with $\mathrm{CFL}=0.8$, and evolve to $t=0.1$.
     Table~\ref{tab:smooth-wave-oe-strength} reports the convergence histories
     without OE and with $s_{\mathrm{OE}}=0.01$, $0.02$, and $0.05$.

     \begin{table}[t]
     \centering
     \scriptsize
     \caption{Relative density errors and observed convergence orders for the
     $\mathbb{Q}_2$ smooth-wave problem at different OE strengths
     $s_{\mathrm{OE}}$.}
     \label{tab:smooth-wave-oe-strength}
     \begin{tabular}{ccrrrrrr}
     \toprule
     $s_{\mathrm{OE}}$ & Mesh & $E_1(\rho)$ & Order & $E_2(\rho)$ & Order &
     $E_\infty(\rho)$ & Order \\
     \midrule
     $0$ (off) & $16^2$  & $2.198\!\times\!10^{-4}$ & --    & $2.809\!\times\!10^{-4}$ & --    & $1.000\!\times\!10^{-3}$ & -- \\
               & $32^2$  & $2.782\!\times\!10^{-5}$ & 2.982 & $3.584\!\times\!10^{-5}$ & 2.971 & $1.264\!\times\!10^{-4}$ & 2.984 \\
               & $64^2$  & $3.503\!\times\!10^{-6}$ & 2.990 & $4.526\!\times\!10^{-6}$ & 2.985 & $1.599\!\times\!10^{-5}$ & 2.983 \\
               & $128^2$ & $4.385\!\times\!10^{-7}$ & 2.998 & $5.673\!\times\!10^{-7}$ & 2.996 & $1.999\!\times\!10^{-6}$ & 2.999 \\
     \midrule
     $0.01$    & $16^2$  & $2.760\!\times\!10^{-4}$ & --    & $3.545\!\times\!10^{-4}$ & --    & $1.187\!\times\!10^{-3}$ & -- \\
               & $32^2$  & $3.124\!\times\!10^{-5}$ & 3.143 & $4.064\!\times\!10^{-5}$ & 3.125 & $1.412\!\times\!10^{-4}$ & 3.071 \\
               & $64^2$  & $3.686\!\times\!10^{-6}$ & 3.083 & $4.808\!\times\!10^{-6}$ & 3.080 & $1.692\!\times\!10^{-5}$ & 3.061 \\
               & $128^2$ & $4.494\!\times\!10^{-7}$ & 3.036 & $5.841\!\times\!10^{-7}$ & 3.041 & $2.057\!\times\!10^{-6}$ & 3.040 \\
     \midrule
     $0.02$    & $16^2$  & $3.728\!\times\!10^{-4}$ & --    & $4.775\!\times\!10^{-4}$ & --    & $1.453\!\times\!10^{-3}$ & -- \\
               & $32^2$  & $3.684\!\times\!10^{-5}$ & 3.339 & $4.795\!\times\!10^{-5}$ & 3.316 & $1.625\!\times\!10^{-4}$ & 3.161 \\
               & $64^2$  & $3.939\!\times\!10^{-6}$ & 3.225 & $5.157\!\times\!10^{-6}$ & 3.217 & $1.803\!\times\!10^{-5}$ & 3.172 \\
               & $128^2$ & $4.612\!\times\!10^{-7}$ & 3.094 & $6.029\!\times\!10^{-7}$ & 3.097 & $2.121\!\times\!10^{-6}$ & 3.087 \\
     \midrule
     $0.05$    & $16^2$  & $1.085\!\times\!10^{-3}$ & --    & $1.411\!\times\!10^{-3}$ & --    & $4.226\!\times\!10^{-3}$ & -- \\
               & $32^2$  & $7.809\!\times\!10^{-5}$ & 3.796 & $1.058\!\times\!10^{-4}$ & 3.738 & $3.273\!\times\!10^{-4}$ & 3.691 \\
               & $64^2$  & $5.197\!\times\!10^{-6}$ & 3.910 & $6.864\!\times\!10^{-6}$ & 3.946 & $2.330\!\times\!10^{-5}$ & 3.812 \\
               & $128^2$ & $5.132\!\times\!10^{-7}$ & 3.340 & $6.738\!\times\!10^{-7}$ & 3.349 & $2.355\!\times\!10^{-6}$ & 3.306 \\
     \bottomrule
     \end{tabular}
     \end{table}

     Without OE, the scheme exhibits the expected third-order convergence in
     all three norms.  The damped calculations show the same behavior: on the
     finest refinement, the observed orders range from $3.04$ to $3.34$ in
     $L^1$, from $3.04$ to $3.35$ in $L^2$, and from $3.04$ to $3.31$ in
     $L^\infty$.  The expected third-order accuracy of the $\mathbb Q_2$
     discretization is therefore retained for every tested damping strength.

     At finite resolution, OE introduces additional dissipation into this
     smooth traveling wave, with a larger effect at stronger damping.  On the
     $16^2$ mesh, $s_{\mathrm{OE}}=0.01$, $0.02$, and $0.05$ increase the $L^1$
     error relative to the undamped value by factors of approximately $1.26$,
     $1.70$, and $4.93$, respectively.  This influence diminishes rapidly under
     refinement: on the $128^2$ mesh, the corresponding factors are $1.02$,
     $1.05$, and $1.17$.  In particular,
     $s_{\mathrm{OE}}=0.02$ adds only about $5$--$6\%$ to the three density
     errors on the finest mesh while retaining third-order convergence.

\subsection{Two-dimensional Riemann problems}
\label{sec:numerics-riemann}

We next consider three four-state Riemann problems reported by Cao, Peng, and
Wu~\cite{CaoPengWu2025} for the Cartesian SRHD system
\eqref{eq:planar-srhd-balance}.  The initial discontinuities at $x=0.5$ and
$y=0.5$ partition the domain into four quadrants.  We impose outflow boundary
conditions and set $\Gamma=5/3$.  The primitive state is denoted by
$Q=(\rho,v_x,v_y,p)^T$, where $v_x$ and $v_y$ are the Cartesian velocity components.
The first problem is initialized by
\begin{equation}
 Q(x,y,0)=
 \begin{cases}
 (0.5,0.5,-0.5,5)^T, & x>0.5,\ y>0.5,\\
 (1,0.5,0.5,5)^T, & x<0.5,\ y>0.5,\\
 (3,-0.5,0.5,5)^T, & x<0.5,\ y<0.5,\\
 (1.5,-0.5,-0.5,5)^T, & x>0.5,\ y<0.5.
 \end{cases}
 \label{eq:numerics-riemann-a}
\end{equation}
This configuration consists of four interacting contact discontinuities
(vortex sheets).  The second problem is initialized by
\begin{equation}
 Q(x,y,0)=
 \begin{cases}
 (0.1,0,0,0.01)^T, & x>0.5,\ y>0.5,\\
 (0.1,0.99,0,1)^T, & x<0.5,\ y>0.5,\\
 (0.5,0,0,1)^T, & x<0.5,\ y<0.5,\\
 (0.1,0,0.99,1)^T, & x>0.5,\ y<0.5.
 \end{cases}
 \label{eq:numerics-riemann-b}
\end{equation}
The third problem is initialized by
\begin{equation}
 Q(x,y,0)=
 \begin{cases}
 (0.1,0,0,20)^T, & x>0.5,\ y>0.5,\\
 (\rho_{\rm R},v_{\rm R},0,0.05)^T,
 & x<0.5,\ y>0.5,\\
 (0.01,0,0,0.05)^T, & x<0.5,\ y<0.5,\\
 (\rho_{\rm R},0,v_{\rm R},0.05)^T,
 & x>0.5,\ y<0.5,
 \end{cases}
 \label{eq:numerics-riemann-c}
\end{equation}
where $\rho_{\rm R}=0.00414329639576$ and
$v_{\rm R}=0.9946418833556542$.
The latter two problems contain ultra-relativistic states and provide stringent
tests of the PCP limiter.

All three problems are computed on $\Omega=[0,1]^2$ using $\mathbb Q_2$
elements on $400\times400$ uniform meshes with ${\rm CFL}=0.8$ and
$s_{\mathrm{OE}}=0.02$, and are evolved to $t=0.4$.
Figure~\ref{fig:riemann-three} shows 30 contours of $\ln\rho$.  In the first
problem, the four vortex sheets roll up into a pronounced spiral around a
low-density core.  The latter two problems develop strongly interacting
curved shocks and contacts from the ultra-relativistic quadrant states.  In
all three cases, the scheme preserves admissibility while resolving the
principal wave structures.  The spiral roll-up in the first problem and the
curved wave patterns in the latter two agree well with the corresponding
results in~\cite{CaoPengWu2025}.

\begin{figure}[t]
 \centering
 \includegraphics[width=\textwidth]
 {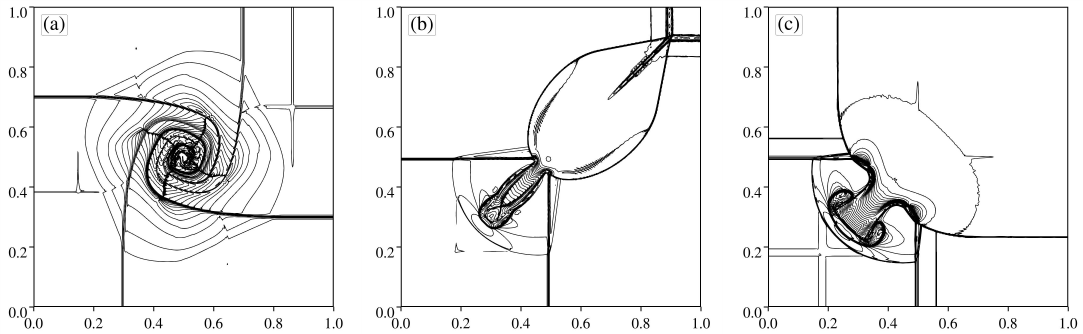}
 \caption{Contours of $\ln\rho$ at $t=0.4$ for the three two-dimensional
 Riemann problems.  Panels (a)--(c) correspond to
 \eqref{eq:numerics-riemann-a}--\eqref{eq:numerics-riemann-c}, respectively.
 Each panel uses 30 uniformly spaced contour levels over the ranges
 $[-4.68,1.19]$, $[-6.0,2.0]$, and $[-8.4,-2.2]$, respectively.  All cases use
 $\mathbb Q_2$ elements on $400\times400$ uniform meshes
 with ${\rm CFL}=0.8$ and $s_{\mathrm{OE}}=0.02$.}
 \label{fig:riemann-three}
\end{figure}

\subsection{Shock--bubble interaction}
\label{sec:numerics-shock-bubble}

The final Cartesian SRHD benchmark considers the interaction of a left-moving
planar shock with either a light or a heavy circular density bubble, following
Duan and Tang~\cite{DuanTang2019}.  The computational domain is
$\Omega=[0,325]\times[-45,45]$, and the shock is initially located at $x=265$.

\begin{figure}[t]
 \centering
 \includegraphics[width=\textwidth]
 {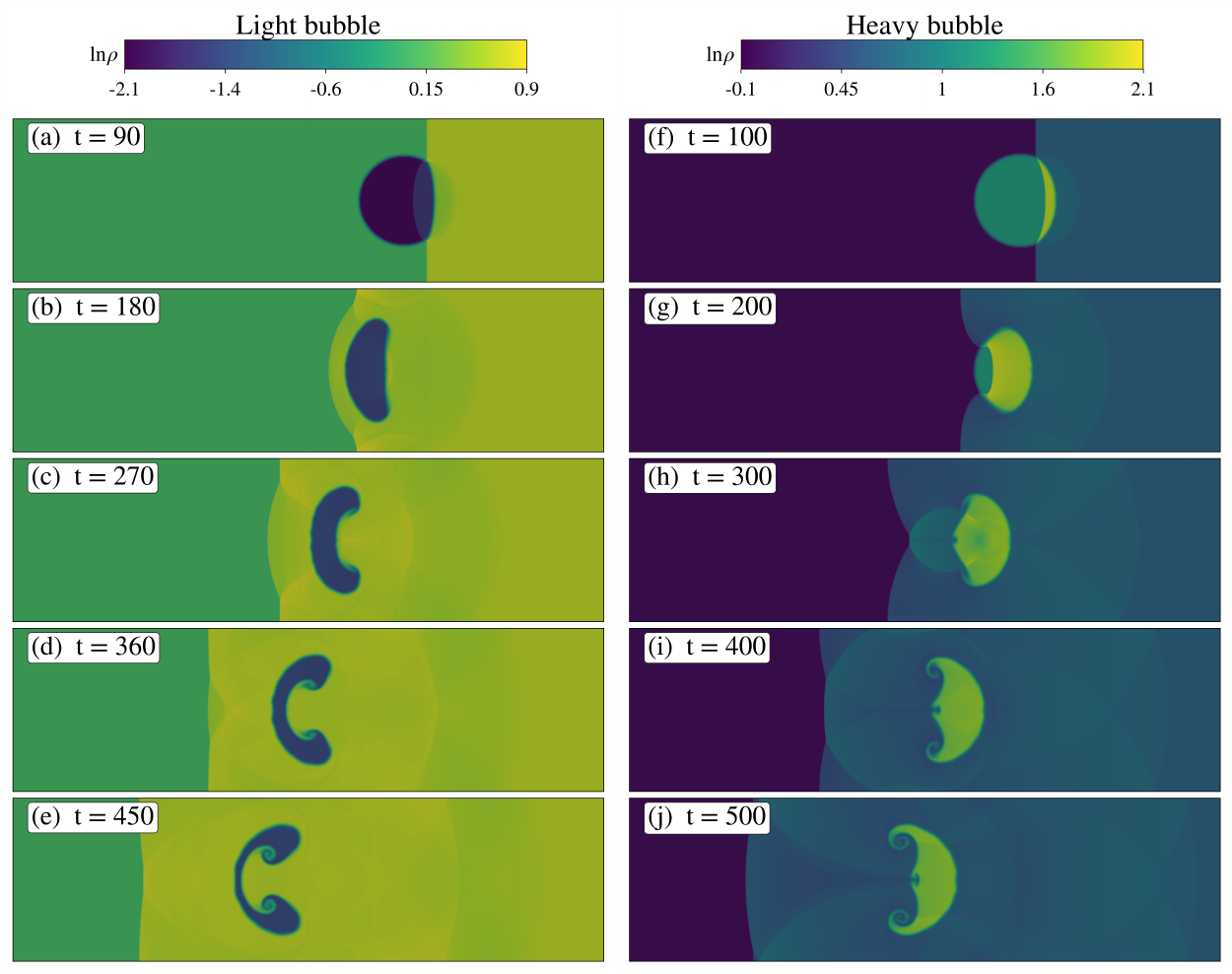}
 \caption{Log-density evolution for the light-bubble problem (panels
 (a)--(e), left column, $t=90,180,270,360,450$) and heavy-bubble problem
 (panels (f)--(j), right column, $t=100,200,300,400,500$), ordered from top
 to bottom.  Both computations use
 $\mathbb Q_2$ elements on $650\times180$ uniform meshes with ${\rm CFL}=0.8$
 and $s_{\mathrm{OE}}=0.02$.
 Independent color scales are shown because the two cases have different
 density ranges.}
 \label{fig:shock-bubble}
\end{figure}

Away from the bubble, the post-shock state is
\[
 Q_{\rm post}
 =(1.865225080631180,-0.196781107378299,0,0.15)^T.
\]
The initial primitive field is
\begin{equation}
 Q(x,y,0)=
 \begin{cases}
 (1,0,0,0.05)^T, & x<265,\\
 Q_{\rm post}, & x>265.
 \end{cases}
 \label{eq:numerics-shock-bubble-shock}
\end{equation}
A circular bubble of radius $25$ centered at $(215,0)$ replaces the ambient
state inside
\begin{equation}
 (x-215)^2+y^2\le25^2.
 \label{eq:numerics-shock-bubble-circle}
\end{equation}
We consider both a light and a heavy bubble,
\begin{equation}
 Q_{\rm bubble}=
 \begin{cases}
 (0.1358,0,0,0.05)^T, & \text{light bubble},\\
 (3.1538,0,0,0.05)^T, & \text{heavy bubble}.
 \end{cases}
 \label{eq:numerics-shock-bubble-states}
\end{equation}
The adiabatic index is $\Gamma=5/3$.  Reflecting conditions are imposed at
$y=\pm45$, the post-shock state is prescribed at the right boundary, and a
transmissive outflow condition is used at the left boundary.

Both cases are computed with $\mathbb Q_2$ elements on a $650\times180$
uniform mesh and ${\rm CFL}=0.8$.  We adopt $s_{\mathrm{OE}}=0.02$ for the
time sequences shown in Figure~\ref{fig:shock-bubble}.
As the shock crosses each bubble, transmitted and reflected waves form, while
the material interface is compressed and subsequently rolls up into vortical
structures.  At later times, both cases develop complex wakes and
smaller-scale waves, while the principal shocks and material interfaces remain
sharply resolved.  The evolution of both the light and heavy bubbles agrees
well with the corresponding results in~\cite{DuanTang2019}.

\FloatBarrier

To assess the finite-resolution sensitivity to the OE damping strength, we
also repeat the final-time calculations with $s_{\mathrm{OE}}=0.01$ and $0.05$
on the same mesh.  Figure~\ref{fig:shock-bubble-oe-strength} compares final-time enlargements for
both bubbles at $s_{\mathrm{OE}}=0.01$, $0.02$, and $0.05$.  The principal wake
geometry and the paired vortical structures remain consistent across the three
values.  Increasing $s_{\mathrm{OE}}$ progressively broadens the material
interfaces and smooths the tightly wound spiral tips.  This effect is most
apparent in the heavy-bubble case, where the smaller inner roll-up is weakened
at $s_{\mathrm{OE}}=0.05$.  Conversely, $s_{\mathrm{OE}}=0.01$ retains the
sharpest interfaces and the most fine-scale structure.
The intermediate value $s_{\mathrm{OE}}=0.02$ provides a balanced combination
of interface resolution and regularization.

\begin{figure}[!htb]
 \centering
 \includegraphics[width=0.9\textwidth]
 {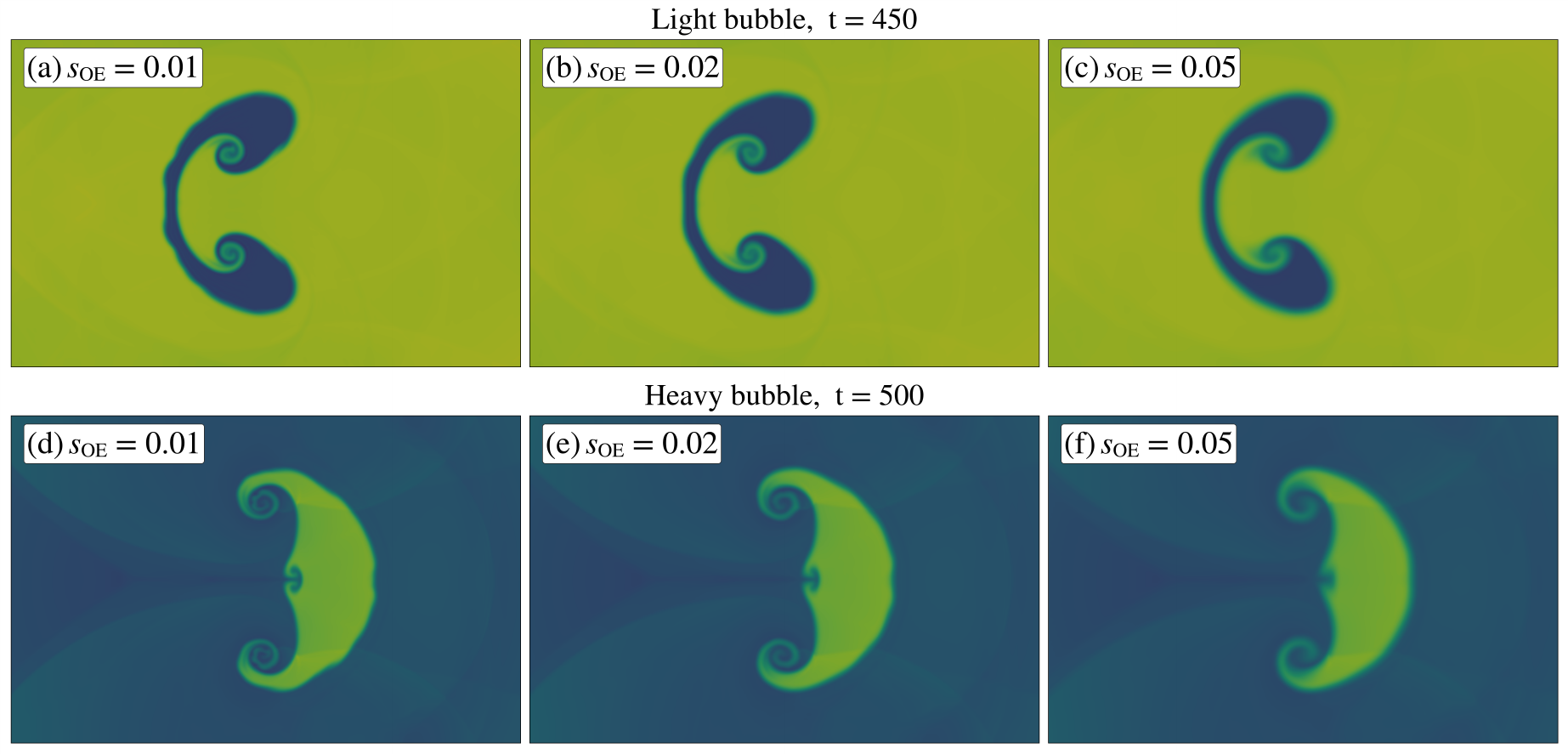}
 \caption{Effect of OE strength on the final-time shock--bubble solutions.
 The upper row shows the light bubble at $t=450$, and the lower row shows
 the heavy bubble at $t=500$.  All panels show $\ln\rho$ over the cropped
 window $[80,220]\times[-45,45]$.  The light- and
 heavy-bubble rows use the same Viridis mappings and ranges as
 Figure~\ref{fig:shock-bubble}, namely $[-2.1,0.9]$ and $[-0.1,2.1]$,
 respectively; the color bars are not repeated.  Each calculation uses
 $\mathbb Q_2$ elements on the same $650\times180$ uniform mesh with
 ${\rm CFL}=0.8$.}
 \label{fig:shock-bubble-oe-strength}
\end{figure}
\FloatBarrier

\subsection{Axisymmetric relativistic jet}
\label{sec:numerics-axisymmetric-jet}

We next consider the pressure-matched C2 axisymmetric ultra-relativistic jet
introduced by Mart\'i et al.~\cite{MartiEtAl1997}.  Its light beam,
relativistic inflow, strong shocks, and shear-driven instabilities make it a
demanding benchmark for the robustness and resolution of axisymmetric SRHD
schemes; see also \cite{WuTang2015,CaoPengWu2025}.  Figure~\ref{fig:axisymmetric-jet-c2}
shows the resulting evolution.

We solve the axisymmetric
SRHD system~\eqref{eq:axisymmetric-srhd-balance} in the meridional half-plane
with coordinates $(r,z)$.  The ambient and beam primitive states are
\begin{equation}
 V_{\rm amb}=(1,0,0,p_j)^T,
 \qquad
 V_{\rm beam}=(0.01,0,0.99,p_j)^T,
 \label{eq:numerics-axisymmetric-jet-states}
\end{equation}
where $V=(\rho,v_r,v_z,p)^T$.  For the C2 model,
$p_j=1.70305\times10^{-4}$, $\Gamma=5/3$, and
$\Omega=[0,15]\times[0,45]$.  At the lower boundary $z=0$, the beam state is
prescribed for $r\leq1$ and the ambient state for $r>1$.  We impose
axisymmetry at $r=0$, the ambient characteristic far-field condition at
$r=15$, and a transmissive condition at $z=45$.

The computation uses $\mathbb Q_2$ elements on a $360\times1080$ affine mesh
with ${\rm CFL}=0.8$ and $s_{\mathrm{OE}}=0.02$.  Because the physical
reduction weight $w=r$ vanishes
at the symmetry axis, cells adjacent to $r=0$ use the axis-cell treatment of
\ref{app:simple-zero-treatment}.

\begin{figure}[!t]
 \centering
 \includegraphics[width=0.91\textwidth]
 {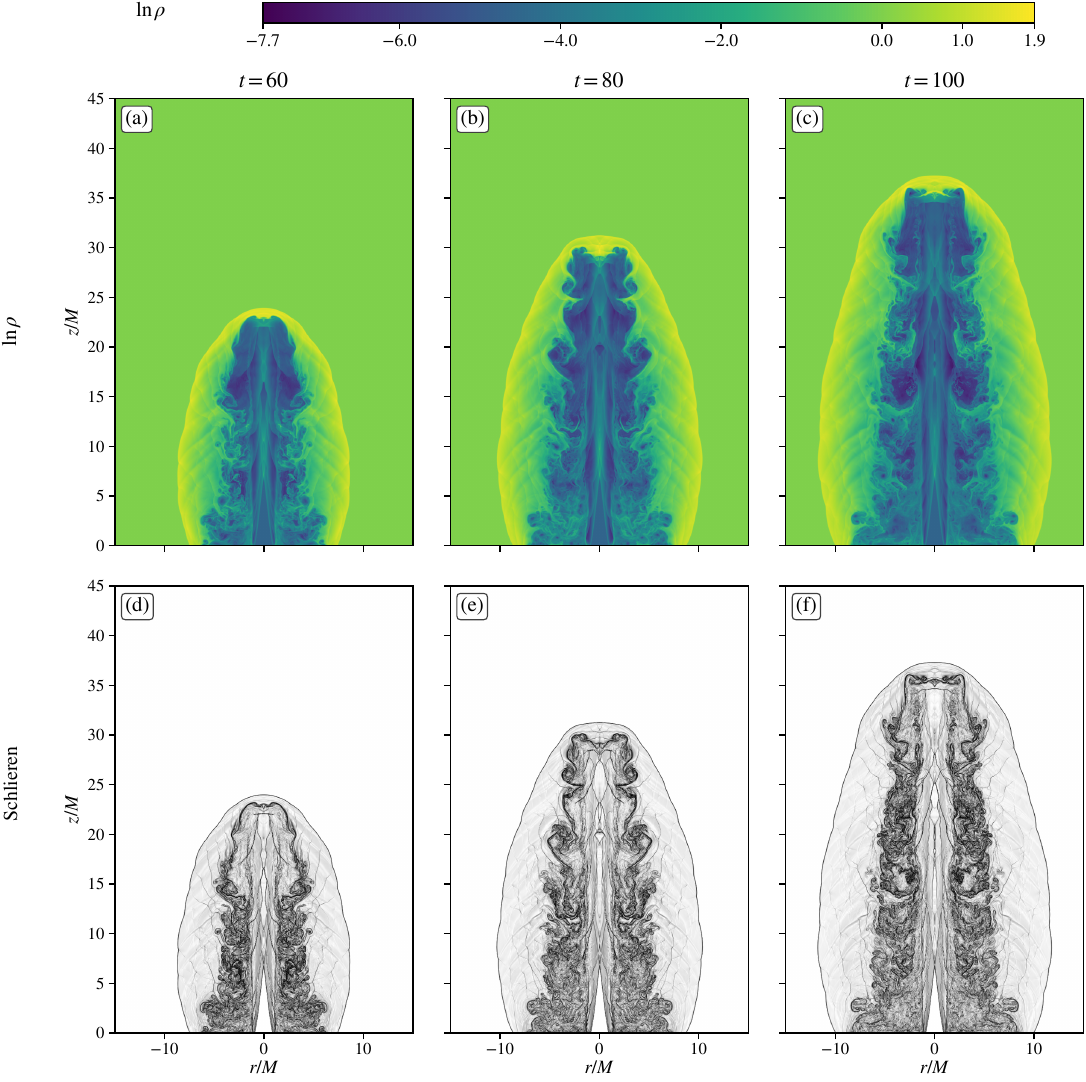}
 \caption{Evolution of the C2 axisymmetric relativistic jet at $t=60$, $80$,
 and $100$ (left to right).  Panels (a)--(c) show $\ln\rho$ using the common
 range $[-7.7,1.9]$ and the Viridis color map.  Panels (d)--(f) show the density schlieren
 diagnostic~\eqref{eq:numerics-axisymmetric-schlieren}.  The
 computation uses $\mathbb Q_2$ elements on a $360\times1080$ affine mesh with
 ${\rm CFL}=0.8$ and $s_{\mathrm{OE}}=0.02$.  The numerical half-plane is
 reflected across the axis for display.}
 \label{fig:axisymmetric-jet-c2}
\end{figure}

For visualization, the numerical half-plane is reflected across $r=0$.  The
upper row of Figure~\ref{fig:axisymmetric-jet-c2} shows $\ln\rho$ over the
fixed display range $[-7.7,1.9]$, while the lower row shows the normalized
density schlieren diagnostic
\begin{equation}
 \mathcal S_\rho
 =\exp\!\left(-\frac{10}{140}\left|\nabla\ln\rho\right|\right).
 \label{eq:numerics-axisymmetric-schlieren}
 \end{equation}
As the jet advances, a bow shock encloses an expanding cocoon, while
shear-driven vortical structures develop along the beam--cocoon interface.
Both diagnostics sharply resolve these features through $t=100$.  The
overall jet-head propagation, cocoon shape, and internal morphology agree
well with the original simulations of Mart\'i et al.~\cite{MartiEtAl1997}
and with subsequent high-order calculations~\cite{WuTang2015,CaoPengWu2025}.
\FloatBarrier

\subsection{Steady Michel accretion}
\label{sec:numerics-michel}

We next assess the accuracy of the Schwarzschild GRHD discretization using the
Michel solution for stationary, spherically symmetric transonic accretion
onto a Schwarzschild black hole~\cite{Michel1972,RadiceRezzolla2011}.  We
solve the Schwarzschild GRHD system~\eqref{eq:schwarzschild-bh-balance} in
the meridional half-plane.  The exact solution is smooth and purely radial throughout
the computational domain, providing a curved-spacetime test of the accuracy with which
the two-dimensional weighted formulation resolves a smooth stationary flow.

We set $M=1$ and $\Gamma=4/3$, choose the critical radius $r_c=8M$, and
normalize the positive rest-mass accretion rate to unity.  These choices
determine the physical transonic branch, which provides both the initial
condition on
$(r,\theta)\in[1.5M,11.5M]\times[0,\pi]$ and the boundary data at the two
radial boundaries.  The inner boundary at $r=1.5M$ lies inside the event
horizon at $r=2M$.  The physical reduction weight is $w=r\sin\theta$, and
cells adjacent to the polar axes are treated as described in
\ref{app:simple-zero-treatment}.

To quantify the departure from the steady state, we initialize the nodal
solution with the exact Michel profile and evolve to $t=2M$ using
$\mathbb{Q}_2$ elements with $\mathrm{CFL}=0.8$.  The relative density errors
defined in~\eqref{eq:numerics-density-errors} are evaluated using the physical
volume element
\begin{equation}
\dd V
=
r^2\sqrt{1+\frac{2M}{r}}\sin\theta\,\dd r\,\dd\theta.
\end{equation}
Since the exact solution is independent of $\theta$, we fix $N_\theta=4$
and refine only in the radial direction, from $N_r=16$ to $128$.
Table~\ref{tab:michel-oe-strength} reports the resulting convergence
histories without OE and with $s_{\mathrm{OE}}=0.01$, $0.02$, and $0.05$.

\begin{table}[t]
\centering
\scriptsize
\caption{Relative density errors and observed convergence orders for the
stationary Michel problem using $\mathbb{Q}_2$ elements at different OE
strengths $s_{\mathrm{OE}}$.}
\label{tab:michel-oe-strength}
\begin{tabular}{ccrrrrrr}
\toprule
$s_{\mathrm{OE}}$ & Mesh & $E_1(\rho)$ & Order & $E_2(\rho)$ & Order &
$E_\infty(\rho)$ & Order \\
\midrule
$0$ (off) & $16\times4$  & $2.734\!\times\!10^{-4}$ & --    & $1.434\!\times\!10^{-3}$ & --    & $4.615\!\times\!10^{-3}$ & -- \\
          & $32\times4$  & $3.745\!\times\!10^{-5}$ & 2.868 & $2.089\!\times\!10^{-4}$ & 2.779 & $8.373\!\times\!10^{-4}$ & 2.462 \\
          & $64\times4$  & $4.864\!\times\!10^{-6}$ & 2.945 & $2.831\!\times\!10^{-5}$ & 2.884 & $1.352\!\times\!10^{-4}$ & 2.631 \\
          & $128\times4$ & $6.284\!\times\!10^{-7}$ & 2.953 & $3.672\!\times\!10^{-6}$ & 2.946 & $2.032\!\times\!10^{-5}$ & 2.734 \\
\midrule
$0.01$    & $16\times4$  & $1.872\!\times\!10^{-4}$ & --    & $1.014\!\times\!10^{-3}$ & --    & $3.657\!\times\!10^{-3}$ & -- \\
          & $32\times4$  & $3.121\!\times\!10^{-5}$ & 2.584 & $1.750\!\times\!10^{-4}$ & 2.535 & $7.648\!\times\!10^{-4}$ & 2.258 \\
          & $64\times4$  & $4.496\!\times\!10^{-6}$ & 2.796 & $2.609\!\times\!10^{-5}$ & 2.746 & $1.305\!\times\!10^{-4}$ & 2.551 \\
          & $128\times4$ & $6.072\!\times\!10^{-7}$ & 2.888 & $3.537\!\times\!10^{-6}$ & 2.883 & $1.994\!\times\!10^{-5}$ & 2.710 \\
\midrule
$0.02$    & $16\times4$  & $1.342\!\times\!10^{-4}$ & --    & $7.293\!\times\!10^{-4}$ & --    & $2.834\!\times\!10^{-3}$ & -- \\
          & $32\times4$  & $2.579\!\times\!10^{-5}$ & 2.379 & $1.477\!\times\!10^{-4}$ & 2.304 & $6.952\!\times\!10^{-4}$ & 2.028 \\
          & $64\times4$  & $4.161\!\times\!10^{-6}$ & 2.632 & $2.419\!\times\!10^{-5}$ & 2.610 & $1.260\!\times\!10^{-4}$ & 2.464 \\
          & $128\times4$ & $5.872\!\times\!10^{-7}$ & 2.825 & $3.415\!\times\!10^{-6}$ & 2.825 & $1.958\!\times\!10^{-5}$ & 2.686 \\
\midrule
$0.05$    & $16\times4$  & $1.634\!\times\!10^{-4}$ & --    & $9.587\!\times\!10^{-4}$ & --    & $1.828\!\times\!10^{-3}$ & -- \\
          & $32\times4$  & $1.837\!\times\!10^{-5}$ & 3.153 & $1.041\!\times\!10^{-4}$ & 3.203 & $5.112\!\times\!10^{-4}$ & 1.839 \\
          & $64\times4$  & $3.315\!\times\!10^{-6}$ & 2.470 & $1.992\!\times\!10^{-5}$ & 2.386 & $1.132\!\times\!10^{-4}$ & 2.174 \\
          & $128\times4$ & $5.330\!\times\!10^{-7}$ & 2.637 & $3.113\!\times\!10^{-6}$ & 2.677 & $1.852\!\times\!10^{-5}$ & 2.612 \\
\bottomrule
\end{tabular}
\end{table}

Without OE, the $E_1$ and $E_2$ rates approach the expected third order,
while the $E_\infty$ rate reaches $2.73$ on the finest refinement.  In
contrast to the traveling-wave test, OE improves the Michel steady-state
accuracy: for every tested strength and mesh, all three errors are smaller
than their undamped counterparts.  By the finest mesh, however, all four
configurations yield comparable errors in each
norm, with a relative spread of only approximately $9$--$15\%$, indicating
reduced sensitivity to the damping strength under refinement.  The mechanism
responsible for this improvement is not investigated here.  Nevertheless,
within the tested range, OE retains high-order convergence while reducing the
departure from the stationary Michel solution.

\subsection{Axisymmetric Schwarzschild Bondi--Hoyle accretion}
\label{sec:numerics-schwarzschild-bondi-hoyle}

We next consider the M1a relativistic Bondi--Hoyle flow of
Lora-Clavijo and Guzm\'an~\cite{LoraClavijoGuzman2013}.  We solve the
axisymmetric Schwarzschild system~\eqref{eq:schwarzschild-bh-balance} in
horizon-penetrating coordinates with $M=1$, $\Gamma=4/3$, and the asymptotic
wind data
\begin{equation}
 \begin{aligned}
  \rho_\infty&=10^{-6},
  &c_{s,\infty}&=0.1,
  &v_\infty&=0.5,\\
  p_\infty&=7.7319587628866007\times10^{-9}.
 \end{aligned}
 \label{eq:numerics-m1a-asymptotic-state}
\end{equation}
The wind travels in the positive $z$ direction and has asymptotic Mach number
$\mathcal M_\infty=v_\infty/c_{s,\infty}=5$.  The associated accretion radius is
$r_{\rm acc}=M/(v_\infty^2+c_{s,\infty}^2)=3.8461538461538458$.
The computational domain in $(r,\theta)$ is
\begin{equation}
 (r,\theta)\in[1.5,10r_{\rm acc}]\times[0,\pi].
 \label{eq:numerics-m1a-domain}
\end{equation}
At the inner boundary, which lies inside the event horizon, we impose a
transmissive excision condition.  At the outer boundary, the asymptotic wind
is prescribed on the upstream portion, while a transmissive condition is
imposed downstream.  The physical reduction weight is $w=r\sin\theta$; cells
adjacent to the two polar axes use the axis-cell treatment of
\ref{app:simple-zero-treatment}.

The computation uses $\mathbb Q_2$ elements on a $256\times128$
uniform $(r,\theta)$ mesh with ${\rm CFL}=0.8$ and
$s_{\mathrm{OE}}=0.02$, and is evolved to $t=500M$.
Figure~\ref{fig:schwarzschild-m1a} displays $\ln\rho$ and the Lorentz factor
$L$ in the central $r\le20M$ region, using the meridional coordinates
$x=r\sin\theta$ and $z=r\cos\theta$.  At the final time, a well-defined
downstream shock cone extends from the horizon and contains gas substantially
denser than the ambient flow.  The Lorentz-factor field shows deceleration
behind the shock and localized acceleration near the horizon.  The resulting
shock-cone morphology agrees qualitatively with the M1a density distribution
reported in~\cite{LoraClavijoGuzman2013}.

\begin{figure}[!ht]
 \centering
 \includegraphics[width=\textwidth]
 {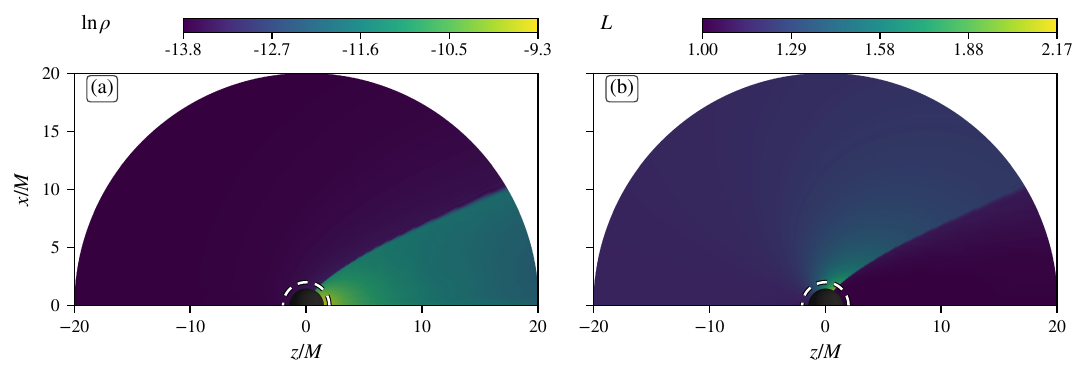}
 \caption{Axisymmetric Schwarzschild M1a Bondi--Hoyle flow at $t=500M$,
 restricted for display to $r\le20M$.  Panel (a) shows $\ln\rho$ over its
 displayed-data range $[-13.78,-9.35]$, while panel (b) shows the Lorentz
 factor $L$ over its displayed-data range $[1.00,2.17]$.
 For a compact presentation, $z/M$ is plotted horizontally and $x/M$
 vertically, where $x=r\sin\theta$ and $z=r\cos\theta$.  The black region denotes
 the excised core $r\le1.5M$, while the dashed semicircle marks the event
 horizon at $r=2M$.  The computed solution is therefore visible on both sides
 of the horizon.  The computation uses $\mathbb Q_2$ elements on a $256\times128$
 uniform $(r,\theta)$ mesh with ${\rm CFL}=0.8$ and
 $s_{\mathrm{OE}}=0.02$.}
\label{fig:schwarzschild-m1a}
\end{figure}
\FloatBarrier

\subsection{Equatorial Kerr--Schild Bondi--Hoyle accretion}
\label{sec:numerics-kerr-bondi-hoyle}

Our final numerical example is the four-spin equatorial Kerr accretion test of
Font, Ib\'a\~nez, and Papadopoulos~\cite{FontIbanezPapadopoulos1999}.  We
solve the equatorial Kerr--Schild system~\eqref{eq:rphi-reduced-balance} in
the active coordinates $(r,\widetilde\phi)$ over
\begin{equation}
 (r,\widetilde\phi)\in[r_{\min},50.9M]\times[0,2\pi),
 \label{eq:numerics-kerr-domain}
\end{equation}
with periodic identification in $\widetilde\phi$.  The common asymptotic data
are
\begin{equation}
 \begin{aligned}
 M&=1, & \Gamma&=\frac53, & \rho_\infty&=1,
 &v_\infty&=0.5,\\
 \mathcal M_\infty&=5, & c_{s,\infty}&=0.1,
 &p_\infty&=\frac6{985}.
 \end{aligned}
 \label{eq:numerics-kerr-asymptotic-data}
\end{equation}
The metric-adapted velocity field is initialized from the uniform asymptotic
wind specified in~\cite{FontIbanezPapadopoulos1999}.  The same data are held
fixed at the outer boundary, while the inner radial boundary is transmissive.

\begin{figure}[!b]
 \centering
 \includegraphics[width=0.78\textwidth]
 {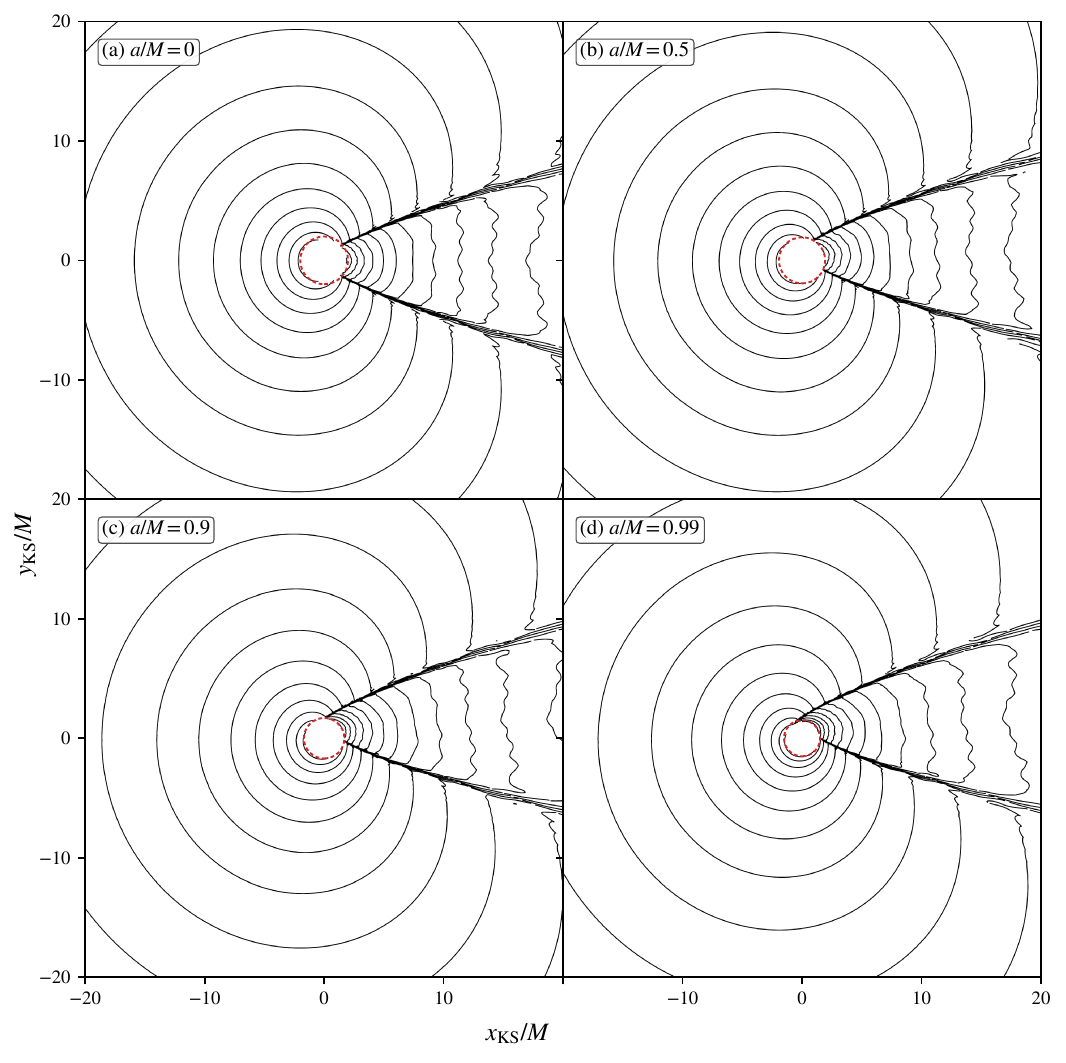}
 \caption{Contours of $\log_{10}(\rho/\rho_\infty)$ for equatorial
 Kerr--Schild Bondi--Hoyle accretion at $t=500M$, displayed in Kerr--Schild Cartesian coordinates over
 $[-20M,20M]^2$.  Panels (a)--(d) correspond to $a/M=0$, $0.5$, $0.9$, and
 $0.99$, respectively.  Each panel contains 20 uniformly spaced levels, with
 a common minimum of $-0.16$ and
 case-dependent maxima of $1.96$, $1.97$, $2.13$, and $2.27$.  The dotted red
 circle marks the event horizon.  All runs use $\mathbb Q_2$ elements on
 $200\times160$ uniform $(r,\widetilde\phi)$ meshes with ${\rm CFL}=0.8$ and
 $s_{\mathrm{OE}}=0.02$.}
 \label{fig:kerr-four-spin-ks-contours}
\end{figure}

The four spin values, excision radii, and horizon radii are
\begin{equation}
\begin{gathered}
 \begin{array}{c|cccc}
 \text{case} & 1 & 2 & 3 & 4 \\ \hline
 a/M        & 0 & 0.5 & 0.9 & 0.99 \\
 r_{\min}/M & 1.8 & 1.8 & 1.4 & 1.0 \\
 r_+/M      & 2 & 1.866025 & 1.435890 & 1.141067
 \end{array},\\[2pt]
 r_+=M+\sqrt{M^2-a^2}.
\end{gathered}
\label{eq:numerics-kerr-cases}
\end{equation}
Thus, the inner boundary lies inside the corresponding event horizon in every
case.

Here $w=r>0$ throughout the computational annulus.  Each case uses $\mathbb Q_2$
elements on a $200\times160$ uniform $(r,\widetilde\phi)$ mesh with
${\rm CFL}=0.8$ and $s_{\mathrm{OE}}=0.02$.  We first examine the solution in
the equatorial Kerr--Schild Cartesian coordinates
$(x_{\rm KS},y_{\rm KS})$, defined by
$x_{\rm KS}+\mathrm{i}y_{\rm KS}=(r+\mathrm{i}a)
\exp(\mathrm{i}\widetilde\phi)$.  Figure~\ref{fig:kerr-four-spin-ks-contours}
shows the normalized log-density over $[-20M,20M]^2$.  A well-defined tail
shock forms in every case, and the large-scale cone changes little with spin;
the rotational influence is concentrated near the black hole.  The shock
opening and downstream contour pattern agree closely with the Kerr--Schild
results in Fig.~3 of \cite{FontIbanezPapadopoulos1999}.
\FloatBarrier

For comparison with the Boyer--Lindquist-coordinate presentation of
\cite{CaoPengWu2025}, Figure~\ref{fig:kerr-four-spin-bl-contours} replots the
same solutions in Boyer--Lindquist Cartesian coordinates.

\begin{figure}[!b]
 \centering
 \includegraphics[width=0.78\textwidth]
 {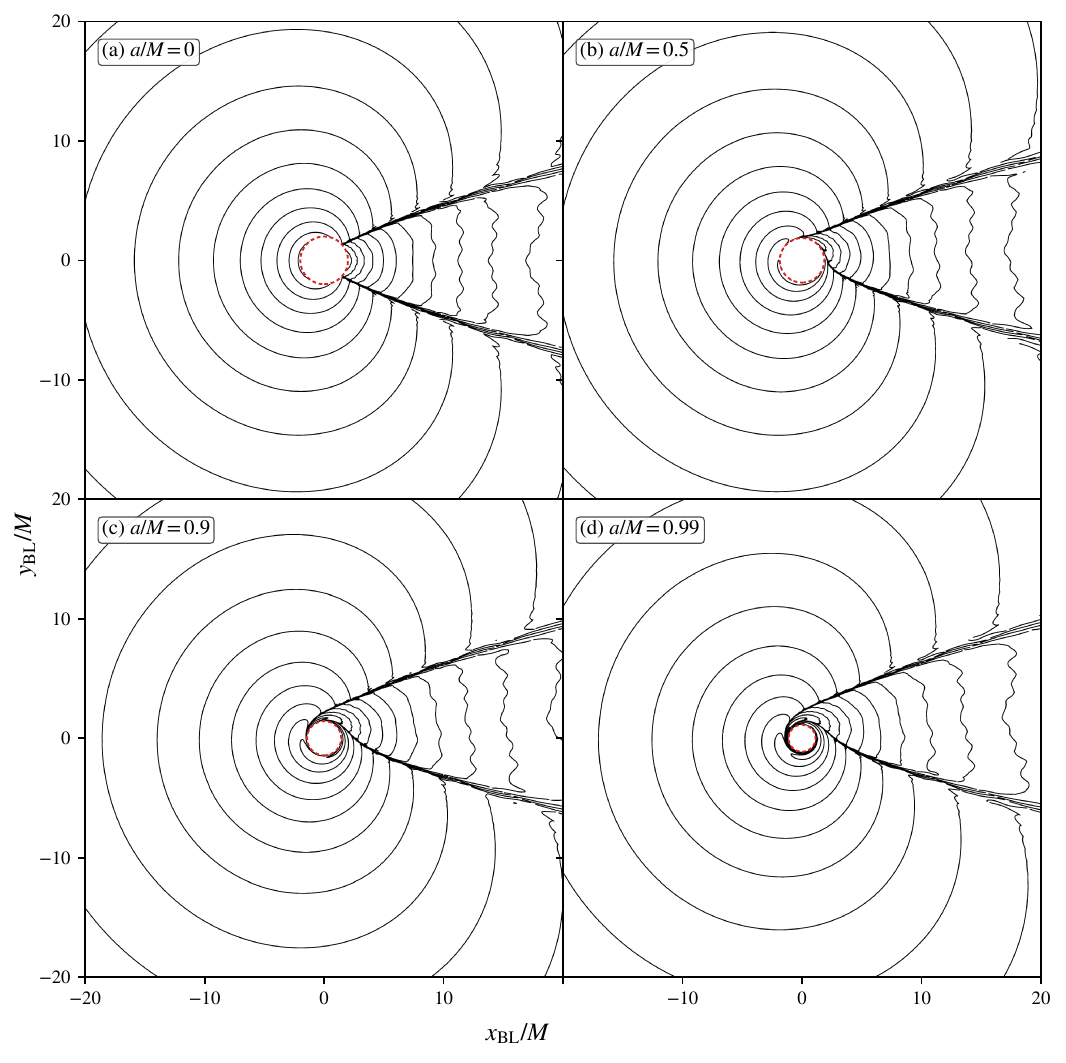}
 \caption{Contours of $\log_{10}(\rho/\rho_\infty)$ for the four Kerr-accretion
 cases in Figure~\ref{fig:kerr-four-spin-ks-contours}, displayed in Boyer--Lindquist
 Cartesian coordinates over $[-20M,20M]^2$.  The corresponding panels in the
 two figures use the same 20 contour levels.  The dotted red circle marks the
 event horizon.}
 \label{fig:kerr-four-spin-bl-contours}
\end{figure}

The plotting coordinates are defined by
\begin{equation}
 \begin{aligned}
 x_{\rm BL}&=r\cos\phi, &
 y_{\rm BL}&=r\sin\phi,\\
 \phi&=\widetilde\phi-\frac{a}{r_+-r_-}
 \log\!\left(\frac{r-r_+}{r-r_-}\right), &
 r_-&=M-\sqrt{M^2-a^2},
 \end{aligned}
 \label{eq:numerics-kerr-plot-map}
\end{equation}
where the additive constant is fixed by
$\phi-\widetilde\phi\to0$ as $r\to\infty$.  Since the Boyer--Lindquist azimuth
is singular at the event horizon, the visualization is restricted to
$r>r_+$. Over the $[-20M,20M]^2$ window, the large-scale cone morphology and the
increasingly pronounced rotational
distortion near the black hole agree qualitatively with Fig.~14 of
\cite{CaoPengWu2025}.
\FloatBarrier

%% file: sections/08_conclusions.tex
\section{Conclusions}\label{sec:conclusions}

We developed a compatible entropy-stable and physical-constraint-preserving
DGSEM for two-dimensional symmetry-reduced GRHD on prescribed stationary
spacetimes.  The local orthonormal transformation separates the standard SRHD
fluid algebra from the stationary geometry, while retaining the physical
reduction weight in the conservative measure.  Discretizing the two-point
flux and geometric source with the same nodal geometry yields the required
flux--source cancellation and a semidiscrete entropy inequality for periodic
problems on conforming affine tensor-product meshes.  The same transformed
variables provide a metric-independent admissible cone, conservative weighted
PCP limiting, and a
regular treatment of simple-zero symmetry axes.  Moreover, a geometry-only
causal speed supplies sufficient dissipation for both the LLF entropy
condition and the PCP splitting.

The fully discrete method combines SSPRK$(3,3)$ time stepping with OE and PCP
scaling.  The numerical
experiments recover high-order accuracy for smooth SRHD and Michel flows and
robustly resolve multidimensional shocks, an axisymmetric relativistic jet,
and Schwarzschild and Kerr accretion flows.  Together, these results support
the proposed framework as a practical high-order method for relativistic
flows on curved stationary backgrounds.

Future work will address curvilinear meshes with compatible discrete metric identities and adaptive mesh refinement with nonconforming interfaces.
Further extensions to general time-dependent
spacetime metrics, fully three-dimensional discretizations, and related
general-relativistic systems---including magnetohydrodynamics, radiation
hydrodynamics, and coupled spacetime--fluid evolution---are also of interest.

%% file: appendices/app_symmetry_reduction.tex
\section{Explicit source terms for the black-hole geometries}
\label{app:black-hole-sources}

\subsection{Axisymmetric Schwarzschild Bondi--Hoyle GRHD}
\label{app:sb-rphi-details}

For the axisymmetric Schwarzschild model \eqref{eq:schwarzschild-bh-balance},
the suppressed-direction contribution $\mathbb S_\perp$ is given in
\eqref{eq:schwarzschild-bh-suppressed-source}. The intrinsic source has the
shared form \eqref{eq:black-hole-active-source}; it remains only to evaluate
$\mathcal S_{\rm V}^{\rm act}$ and $\partial_r\mathcal L$ for the
Schwarzschild geometry.

After removing the azimuthal momentum equation, write
\begin{equation}
 \mathcal S_{\rm V}^{\rm act}
 =(0,R_r^{\rm act},0,R_E^{\rm act})^T.
 \label{eq:app-schwarzschild-active-valencia-source}
\end{equation}
Let $\mathcal I_{\rm act}=\{0,r,\theta\}$. From
\eqref{eq:valencia-source}, define
\begin{align}
 R_r^{\rm act}
 &:=\frac12\sum_{\mu,\nu\in\mathcal I_{\rm act}}
 T^{\mu\nu}\partial_r g_{\mu\nu},
 \label{eq:app-schwarzschild-active-r-source-definition}\\
 R_E^{\rm act}
 &:=T^{r0}\partial_r\alpha
 -\alpha\sum_{\mu,\nu\in\mathcal I_{\rm act}}
 T^{\mu\nu}\Gamma^0_{\mu\nu}.
 \label{eq:app-schwarzschild-active-E-source-definition}
\end{align}
The active spacetime block is independent of $\theta$, so the active
$\theta$-momentum source vanishes.

For the Eddington--Finkelstein Schwarzschild metric in
\eqref{eq:schwarzschild-bh-lapse-shift}--
\eqref{eq:schwarzschild-bh-spatial-metric}, evaluation of the metric
derivatives gives
\begin{equation}
 R_r^{\rm act}
 =-\frac{M}{r^2}\left(T^{00}+2T^{0r}+T^{rr}\right)
  +rT^{\theta\theta}.
 \label{eq:app-schwarzschild-active-radial-source}
\end{equation}
Moreover,
\[
 \partial_r\alpha=\frac{\alpha M}{r(r+2M)},
\]
and the active energy source is
\begin{align}
 R_E^{\rm act}
 ={}&\frac{\alpha M}{r(r+2M)}T^{r0}
 \notag\\
 &-\frac{2M\alpha}{r^3}
 \left[
   M T^{00}+(r+2M)T^{0r}
   +(r+M)T^{rr}-r^3T^{\theta\theta}
 \right].
 \label{eq:app-schwarzschild-active-energy-source}
\end{align}

For the orthonormal factor in \eqref{eq:schwarzschild-bh-frame},
\begin{equation}
 \partial_r\Theta
 =\begin{pmatrix}
  \dfrac{\alpha M}{r(r+2M)}&0\\[2mm]
  0&-r^{-2}
 \end{pmatrix},
 \qquad
 \partial_r\mathcal L
 =\operatorname{diag}(0,\partial_r\Theta,0).
 \label{eq:app-schwarzschild-L-derivative}
\end{equation}
Substitution of \eqref{eq:app-schwarzschild-active-valencia-source},
\eqref{eq:app-schwarzschild-active-r-source-definition}, \eqref{eq:app-schwarzschild-active-E-source-definition}, and
\eqref{eq:app-schwarzschild-L-derivative} into
\eqref{eq:black-hole-active-source} determines $\mathcal S_W$ for the
Schwarzschild model. Together with
\eqref{eq:schwarzschild-bh-suppressed-source}, this completely specifies the
source in \eqref{eq:schwarzschild-bh-balance}.

\subsection{Equatorial Kerr--Schild GRHD}
\label{app:kerr-rphi-details}

For the equatorial Kerr model \eqref{eq:rphi-reduced-balance}, the
suppressed-direction contribution $\mathbb S_\perp$ is already given
explicitly in \eqref{eq:kerr-rphi-suppressed-source}. Hence, by
\eqref{eq:black-hole-active-source}, it remains only to evaluate the active
Valencia source $\mathcal S_{\rm V}^{\rm act}$ and the radial frame derivative
$\partial_r\mathcal L$.

After removing the polar momentum equation, write
\begin{equation}
 \mathcal S_{\rm V}^{\rm act}
 =(0,R_r^{\rm act},0,R_E^{\rm act})^T.
 \label{eq:app-kerr-active-valencia-source}
\end{equation}
Let $\mathcal I_{\rm act}=\{0,r,\widetilde\phi\}$ denote the active spacetime
indices. From \eqref{eq:valencia-source}, define
\begin{align}
 R_r^{\rm act}
 &:=\frac12\sum_{\mu,\nu\in\mathcal I_{\rm act}}
 T^{\mu\nu}\partial_r g_{\mu\nu},
 \label{eq:app-kerr-active-r-source-definition}\\
 R_E^{\rm act}
 &:=T^{r0}\partial_r\alpha
 -\alpha\sum_{\mu,\nu\in\mathcal I_{\rm act}}
 T^{\mu\nu}\Gamma^0_{\mu\nu}.
 \label{eq:app-kerr-active-E-source-definition}
\end{align}
Since the metric is independent of $\widetilde\phi$, the azimuthal momentum
source vanishes.

Evaluating the metric derivatives gives
\begin{align}
 R_r^{\rm act}
 ={}&-\frac{M}{r^2}
 \left(T^{00}+2T^{0r}+T^{rr}\right)
 \notag\\
 &+\frac{2Ma}{r^2}
 \left(T^{0\widetilde\phi}+T^{r\widetilde\phi}\right)
 +\left(r-\frac{Ma^2}{r^2}\right)
 T^{\widetilde\phi\widetilde\phi}.
 \label{eq:app-kerr-active-radial-source}
\end{align}
Moreover,
\[
 \partial_r\alpha=\frac{\alpha M}{r(r+2M)},
\]
and the active energy source becomes
\begin{align}
 R_E^{\rm act}
 ={}&\frac{\alpha M}{r(r+2M)}T^{r0}
 -\frac{2M\alpha}{r^3}\Bigl[
 M T^{00}+(r+2M)T^{0r}
 \notag\\
 &\quad
 -2MaT^{0\widetilde\phi}
 +(r+M)T^{rr}
 -a(r+2M)T^{r\widetilde\phi}
 \notag\\
 &\quad
 +(Ma^2-r^3)T^{\widetilde\phi\widetilde\phi}
 \Bigr].
 \label{eq:app-kerr-active-energy-source}
\end{align}

It remains to evaluate the derivative of the local frame. For the
orthonormal factor $\Theta$ in \eqref{eq:kerr-rphi-theta},
\[
 \partial_r\mathcal K=-\frac{2M}{r^2},
 \qquad
 \partial_r\mathcal B=2r-\frac{2Ma^2}{r^2},
\]
and the nonzero entries of $\partial_r\Theta$ are
\begin{align}
 \partial_r\Theta_{11}
 &=\Theta_{11}\left(
   \frac{\partial_r\mathcal B}{2\mathcal B}
  -\frac{\partial_r\mathcal K}{2\mathcal K}
  -\frac1r\right),
 \label{eq:app-kerr-theta11-derivative}\\
 \partial_r\Theta_{12}
 &=\Theta_{12}\left(
   \frac{\partial_r\mathcal K}{2\mathcal K}
  -\frac{\partial_r\mathcal B}{2\mathcal B}
  -\frac1r\right),
 \label{eq:app-kerr-theta12-derivative}\\
 \partial_r\Theta_{22}
 &=-\Theta_{22}\frac{\partial_r\mathcal B}{2\mathcal B}.
 \label{eq:app-kerr-theta22-derivative}
\end{align}
Therefore
\begin{equation}
 \partial_r\mathcal L
 =\operatorname{diag}(0,\partial_r\Theta,0).
 \label{eq:app-kerr-L-derivative}
\end{equation}
Substitution of \eqref{eq:app-kerr-active-valencia-source},
\eqref{eq:app-kerr-active-r-source-definition},
\eqref{eq:app-kerr-active-E-source-definition}, and
\eqref{eq:app-kerr-L-derivative} into
\eqref{eq:black-hole-active-source} determines $\mathcal S_W$. Together with
$\mathbb S_\perp$ in \eqref{eq:kerr-rphi-suppressed-source}, this completely
specifies the source in the equatorial Kerr system
\eqref{eq:rphi-reduced-balance}.

%% file: appendices/app_local_ec.tex
\section{Local SRHD entropy-conservative states and fluxes}
\label{app:local-ec}

For positive \(x_L,x_R\), define the logarithmic mean
\[
 x^{\ln}=\frac{x_R-x_L}{\log x_R-\log x_L},
\]
evaluated with a stable near-equal formula, and use an overbar for the
arithmetic mean.  Let
\[
 \zeta=\rho/p,\qquad K_a=L\vhat_a,\qquad
 \mathfrak a=1+\frac{1}{(\Gamma-1)\zeta^{\ln}}.
\]

\subsection{Temporal EC state}

Set
\[
 Q_t=\overline L^{\,2}-\overline K_a\overline K_a.
\]
The symmetric temporal state is
\begin{align}
 D^\#&=\rho^{\ln}\overline L,\label{eq:temporal-D}\\
 E^\#&=
 \frac{\mathfrak a\rho^{\ln}\overline L^{\,2}
 +(\overline\rho/\overline\zeta)
 \overline K_a\overline K_a}{Q_t},\label{eq:temporal-E}\\
 m_a^\#&=\frac{\overline K_a}{\overline L}
 \left(E^\#+\frac{\overline\rho}{\overline\zeta}\right).
 \label{eq:temporal-m}
\end{align}
Then
\[
 \uhat_{LR}^{\#}=(D^\#,m_1^\#,\ldots,m_d^\#,E^\#)^T.
\]
The identity \(L^2-K_aK_a=1\) and the entropy-variable jump relations give
\(\jump V^T\uhat_{LR}^{\#}=\jump D\).  At equal states,
\(\uhat_{LR}^{\#}=\uhat\).

\subsection{Directional spatial EC flux}

Define
\[
 \overline v_a=\mean{\vhat_a},\qquad
 \overline K_a=\mean{L\vhat_a},\qquad
 \overline L=\mean L,
\]
\[
 \chi_a=
 \frac{\vhat_{a,L}+\vhat_{a,R}}
 {L_L^{-1}L_R^{-1}(L_L^{-1}+L_R^{-1})},
\]
\[
 \mathcal R=\overline L^{\,2}
 +\overline L\,\overline v_b\chi_b
 -\overline K_b\chi_b,\qquad
 \mathcal Q=\overline\zeta\,\mathcal R.
\]
For local flux direction \(a\), set
\begin{align}
 f_D^a&=\rho^{\ln}\overline K_a,\label{eq:ec-mass}\\
 f_{\mhat_b}^a&=
 \frac{\mathfrak a\,\overline\zeta\,\chi_b f_D^a
 +\overline\rho(\delta_{ab}\mathcal R+
 \overline K_a\chi_b)}{\mathcal Q},\label{eq:ec-momentum}\\
 f_E^a&=
 \frac{\mathfrak a f_D^a+
 \overline K_bf_{\mhat_b}^a}{\overline L}.
 \label{eq:ec-energy}
\end{align}
The vector assembled from these components is symmetric, consistent, and
satisfies
\[
 \jump V^T\fhat_{LR}^{a,\EC}=\jump{D\vhat_a}.
\]
Only contractions over the orthonormal index occur, so the formula has the
same form in two and three dimensions.  The formulas require admissible
endpoint primitives and stable evaluation of every logarithmic mean.

%% file: appendices/app_secondary_algebra.tex
\section{Flux-adapted geometry variables}
\label{app:geometry-state-map}

Let
\begin{equation}
 X=(\alpha,\beta^1,\beta^2,
 \gamma_{11},\gamma_{12},\gamma_{22})^T
 \label{eq:dg-adm-geometry-state}
\end{equation}
denote the active ADM geometry.  The physical reduction weight $w$ is a
separate prescribed scalar field.  In two active dimensions the
upper-triangular Cholesky convention gives
\[
 \det\Theta=\frac1{j_2}.
\]
The flux-adapted geometry state is
\begin{equation}
 Y=(j_2,A_{11},A_{12},A_{22},C^1,C^2)^T,
 \label{eq:dg-y-geometry-state-app}
\end{equation}
where
\begin{equation}
 j_2=\sqrt{\det(\gamma_{IJ})},
 \qquad A=\alpha j_2\Theta,
 \qquad C^I=j_2\beta^I.
 \label{eq:dg-geometry-forward-map}
\end{equation}
Since
\[
 \det A=(\alpha j_2)^2\det\Theta=\alpha^2j_2,
\]
the inverse map is
\begin{equation}
 \alpha=\sqrt{\frac{\det A}{j_2}},
 \qquad \beta^I=\frac{C^I}{j_2},
 \qquad \Theta=\frac{A}{\alpha j_2},
 \qquad
 \gamma=(\Theta^T\Theta)^{-1}.
 \label{eq:dg-geometry-inverse-map}
\end{equation}
Thus $X\leftrightarrow Y$ is a smooth one-to-one local change of the active
geometry coordinates for $j_2>0$.
The complete flux coefficients are formed only in the forward direction,
\begin{equation}
 \begin{aligned}
  \mathcal A&=wA,
  &\mathcal C&=wC,\\
  d_I\mathcal A&=w\,d_IA+A\,d_Iw,
  &d_I\mathcal C&=w\,d_IC+C\,d_Iw.
 \end{aligned}
 \label{eq:dg-complete-coefficient-differentials}
\end{equation}

For later use, let $d_I Y$ be any directional increment of the intrinsic
geometry state.  Differentiating \eqref{eq:dg-geometry-inverse-map} gives
\begin{align}
 d_I\log\alpha
 &=\frac12\left[
   \operatorname{tr}(A^{-1}d_IA)
   -\frac{d_Ij_2}{j_2}\right],
 \label{eq:dg-directional-alpha}\\
 d_I\beta^J
 &=\frac{d_IC^J}{j_2}
   -\beta^J\frac{d_Ij_2}{j_2},
 \label{eq:dg-directional-beta}\\
 d_I\Theta
 &=\frac{d_IA}{\alpha j_2}
   -\Theta\left(\frac{d_I\alpha}{\alpha}
                 +\frac{d_Ij_2}{j_2}\right),
 \label{eq:dg-directional-theta}\\
 d_I\gamma
 &=-\gamma\left[(d_I\Theta)^T\Theta
                  +\Theta^T(d_I\Theta)\right]\gamma.
 \label{eq:dg-directional-gamma}
\end{align}
These are algebraic differentials of the local geometry map.

%% file: appendices/app_llf_speed_bound.tex
\section{Causal bound for classical-LLF entropy stability}
\label{app:llf-speed-bound}
\renewcommand{\thedefinition}{\Alph{section}.\arabic{definition}}

We prove Proposition~\ref{thm:pcp-dominates-entropy-speed}, namely that the
geometry-only causal speed in~\eqref{eq:dg-causal-normal-speed} is sufficient
to make the classical LLF flux entropy stable. Fix a positive-weight face
node and its single-valued stationary geometry. For the fixed face covector
$n$, write
\[
U:=\mathbb W,
\qquad
F(U):=\mathbb H_n(U),
\qquad
\eta(U):=\eta_w(U),
\qquad
\Psi(U):=\Psi_{w,n}(U),
\]
and let $V=\nabla_U\eta$ be the entropy variables. For two admissible states,
set
\begin{equation}
\mathcal C_{LR}
:=\jump{V}^T\frac{F_L+F_R}{2}-\jump{\Psi},
\qquad
\mathcal D_{LR}
:=\jump{V}^T\jump{U}.
\label{eq:llf-CD}
\end{equation}
The entropy defect of the classical LLF flux is therefore
\begin{equation}
\jump{V}^T\widehat{\mathbb H}^{\LLF}_n-\jump{\Psi}
=
\mathcal C_{LR}-\frac{a_f}{2}\mathcal D_{LR}.
\label{eq:llf-entropy-defect}
\end{equation}
Since $\eta$ is strictly convex,
$\mathcal D_{LR}>0$ whenever $U_L\ne U_R$. Thus it suffices to prove
\begin{equation}
|\mathcal C_{LR}|
\le
\frac{a_n^{\rm PCP}}{2}\mathcal D_{LR}.
\label{eq:llf-target-bound}
\end{equation}

\begin{lemma}[Convex entropy-variable domain]
\label{lem:entropy-variable-domain}
For the Gamma-law RHD entropy with $1<\Gamma\le2$, the entropy variables
associated with admissible local orthonormal states satisfy
\begin{equation}
\mathcal V
=
\left\{
V=(V_D,V_m,V_E):
V_D\in\mathbb R,\quad -V_E>|V_m|
\right\},
\label{eq:entropy-variable-domain}
\end{equation}
and hence $\mathcal V$ is convex.
\end{lemma}

\begin{proof}
From~\eqref{eq:entropy-variables},
\[
V_m=\zeta L\widehat v,
\qquad
V_E=-\zeta L,
\qquad
\zeta=\frac{\rho}{p}>0,
\]
so every admissible state satisfies $-V_E>|V_m|$. Conversely, this inequality
gives
\[
\widehat v=-\frac{V_m}{V_E},
\qquad
\zeta=\sqrt{V_E^2-|V_m|^2}>0.
\]
The remaining component determines
\[
s=\Gamma-(\Gamma-1)(V_D-\zeta),
\]
and hence uniquely determines positive $\rho$ and $p$ through
$s=\log p-\Gamma\log\rho$ and $\zeta=\rho/p$. Thus
\eqref{eq:entropy-variable-domain} is the exact entropy-variable image of the
admissible set. It is the product of $\mathbb R$ and the interior of a
Lorentz cone, and is therefore convex.
\end{proof}

\begin{lemma}[Entropy defect along an entropy-variable path]
\label{lem:entropy-threshold-characteristic-bound}
Let $U_L,U_R$ be admissible states at fixed geometry and define
\[
V(\theta)=V_L+\theta\jump{V},
\qquad
U(\theta)=U(V(\theta)),
\qquad
0\le\theta\le1.
\]
Then
\begin{equation}
|\mathcal C_{LR}|
\le
\frac12
\left(
\max_{0\le\theta\le1}
\rho\!\left(\mathsf J_F(\theta)\right)
\right)
\mathcal D_{LR},
\qquad
\mathsf J_F(\theta)
:=\frac{\partial F}{\partial U}(U(\theta)).
\label{eq:entropy-threshold-path-bound}
\end{equation}
\end{lemma}

\begin{proof}
By Lemma~\ref{lem:entropy-variable-domain}, the straight segment $V(\theta)$
remains in the admissible entropy-variable domain. Entropy compatibility and
$\Psi=V^TF-q_\eta$ give
\[
\dd\Psi=F^T\dd V.
\]
Hence, with
\[
g(\theta):=\jump{V}^TF(U(\theta)),
\]
we have $\jump{\Psi}=\int_0^1g(\theta)\,\dd\theta$, and integration by parts
yields the trapezoidal-error identity
\begin{equation}
\mathcal C_{LR}
=
\int_0^1
\left(\theta-\frac12\right)g'(\theta)\,\dd\theta.
\label{eq:llf-C-path}
\end{equation}

Let
\[
K_\eta(V)
:=\frac{\partial U}{\partial V}
=\left(\nabla_U^2\eta\right)^{-1}.
\]
Strict convexity gives $K_\eta=K_\eta^T>0$, while
\[
\mathsf J_FK_\eta
=
\frac{\partial F}{\partial V}
=
\nabla_V^2\Psi
\]
is symmetric. Therefore
$K_\eta^{-1/2}\mathsf J_FK_\eta^{1/2}$ is symmetric and similar to
$\mathsf J_F$, and consequently
\begin{equation}
\left|
z^T\mathsf J_FK_\eta z
\right|
\le
\rho(\mathsf J_F)\,z^TK_\eta z
\qquad
\text{for all }z.
\label{eq:entropy-symmetrizer-rayleigh}
\end{equation}
Since
\[
g'(\theta)
=
\jump{V}^T
\mathsf J_F(\theta)K_\eta(\theta)
\jump{V}
\]
and
\begin{equation}
\mathcal D_{LR}
=
\int_0^1
\jump{V}^TK_\eta(\theta)\jump{V}\,\dd\theta,
\label{eq:llf-D-path}
\end{equation}
using $|\theta-\tfrac12|\le\tfrac12$ in~\eqref{eq:llf-C-path} and then
\eqref{eq:entropy-symmetrizer-rayleigh} proves
\eqref{eq:entropy-threshold-path-bound}.
\end{proof}

\paragraph{Proof of Proposition~\ref{thm:pcp-dominates-entropy-speed}.}
Let
\[
A_{an}:=A_{aI}n_I,
\qquad
C_n:=C^In_I.
\]
At fixed positive geometry,
\[
U=wj_2\widehat U,
\qquad
F(U)
=
w\left(
A_{an}\widehat F^a(\widehat U)
-C_n\widehat U
\right),
\]
and therefore
\begin{equation}
\mathsf J_F
=
\frac1{j_2}
\left(
A_{an}\frac{\partial\widehat F^a}{\partial\widehat U}
-C_n I_4
\right).
\label{eq:llf-normal-jacobian}
\end{equation}
Set
\[
|A_n|:=\sqrt{A_{an}A_{an}},
\qquad
\widehat n_a:=\frac{A_{an}}{|A_n|}.
\]
By rotational invariance of SRHD, the eigenvalues of
\eqref{eq:llf-normal-jacobian} are
\begin{equation}
\lambda_k(\mathsf J_F)
=
\frac{|A_n|\widehat\lambda_k(\widehat n)-C_n}{j_2},
\label{eq:llf-normal-eigenvalues}
\end{equation}
where $\widehat\lambda_k(\widehat n)$ are the local orthonormal SRHD
characteristic speeds in direction $\widehat n$. For the Gamma-law gas,
\[
c_s^2=\frac{\Gamma p}{\rho h}
<\Gamma-1\le1,
\]
so causality implies
$|\widehat\lambda_k(\widehat n)|\le1$. Using the intrinsic-geometry
identities
\[
\frac{|C_n|}{j_2}
=
|\beta^In_I|,
\qquad
\frac{|A_n|}{j_2}
=
\alpha\sqrt{\gamma^{IJ}n_In_J},
\]
we obtain, for every state along the entropy-variable path,
\begin{equation}
\rho(\mathsf J_F)
\le
|\beta^In_I|
+\alpha\sqrt{\gamma^{IJ}n_In_J}
=
a_n^{\rm PCP}.
\label{eq:llf-jacobian-causal-bound}
\end{equation}
Combining Lemma~\ref{lem:entropy-threshold-characteristic-bound} with
\eqref{eq:llf-jacobian-causal-bound} gives
\[
|\mathcal C_{LR}|
\le
\frac{a_n^{\rm PCP}}2\mathcal D_{LR}.
\]
Substitution into~\eqref{eq:llf-entropy-defect} with
$a_f=a_n^{\rm PCP}$ yields
\[
\jump{V}^T\widehat{\mathbb H}^{\LLF}_n-\jump{\Psi}
\le0,
\]
which proves the desired classical-LLF entropy inequality.
\hfill$\square$

%% file: appendices/app_simple_zero_treatment.tex
\section{Axis cells with vanishing reduction weight}
\label{app:simple-zero-treatment}

The analysis in Sections~\ref{sec:dgsem} and~\ref{sec:pcp} assumes $w>0$ at
the solution nodes.  In the axisymmetric models used in the numerical
experiments, however, the reduction weight vanishes on a symmetry axis: at
$r=0$ for the C2 jet and at $\theta=0,\pi$ for the Schwarzschild
calculations.  The weighted state remains regular at these locations.  Indeed,
although $\mathbb W=\chi\widehat U$ vanishes with
$\chi:=wj_2$, the underlying local orthonormal state $\widehat U$ generally
has a finite nonzero limit.  We describe here the discrete treatment used to
preserve this structure.

Let $\xi^I$ be the reference coordinate normal to the axis and let
$\xi^I=\xi^I_\star$ denote an axis GLL node.  Assume that $w$ has a simple
zero there and that $j_2$ is smooth and positive.  For a regular local state
$\widehat U$, the relation $\mathbb W=\chi\widehat U$ gives
\begin{align}
 \chi_\star&=0,\qquad
 (\partial_{\xi^I}\chi)_\star
 =j_{2,\star}(\partial_{\xi^I}w)_\star\ne0,\notag\\
 \chi(\xi^I)
 &=(\partial_{\xi^I}\chi)_\star(\xi^I-\xi^I_\star)
   +\mathcal O((\xi^I-\xi^I_\star)^2),\notag\\
 \mathbb W(\xi^I)
 &=(\partial_{\xi^I}\chi)_\star\widehat U_\star
   (\xi^I-\xi^I_\star)
   +\mathcal O((\xi^I-\xi^I_\star)^2),\notag\\
 \widehat U_\star
 &=\lim_{\xi^I\to\xi^I_\star}\frac{\mathbb W}{\chi}
 =\left.\frac{\partial_{\xi^I}\mathbb W}
                 {\partial_{\xi^I}\chi}\right|_\star.
 \label{eq:app-simple-zero-assumption}
\end{align}
Thus the densitized state has the exact axis value $\mathbb W_\star=0$, while
the finite intrinsic state is determined by a derivative ratio.

\paragraph{Regular discrete axis trace}
At the axis node we retain the exact stored condition $\mathbb W_\star=0$.
Whenever a finite local state is required, we recover it by the discrete
analogue of the derivative ratio in
\eqref{eq:app-simple-zero-assumption}, using the normal SBP derivative along
the same tensor-product nodal line:
\begin{equation}
 \mathbb W_\star=0,\qquad
 \mathcal T_\star^{\widehat U}(\mathbb W_h)
 :=\frac{(\mathsf D_I\mathbb W_h)_\star}
         {(\mathsf D_I\chi)_\star}.
 \label{eq:app-simple-zero-trace}
\end{equation}
The quantity $\mathcal T_\star^{\widehat U}$ is an auxiliary reconstruction
of the regular local state; it does not replace the stored axis value.  It is
used in adjacent volume and source evaluations and in stage stabilization.
Since the complete flux coefficients $\mathcal A=wA$ and $\mathcal C=wC$
also vanish on the axis, the weighted axis flux is zero and no axis Riemann
problem is formed.

\paragraph{Conservative axis repair}
A raw explicit stage can produce a nonzero axis value before stage
stabilization is applied.  Before constructing the regular trace and applying
OE or PCP scaling, we therefore restore $\mathbb W_\star=0$ conservatively on
each normal nodal line.  Let $\mathcal Z$ and $\mathcal P$ denote the sets of
zero- and positive-weight nodes, respectively, and let $\mu_i$ be the mapped
GLL mass weights.  Define
\begin{align}
 c^{\rm rep}
 &:=\frac{\sum_{z\in\mathcal Z}\mu_z\mathbb W_z}
          {\sum_{i\in\mathcal P}\mu_i},\qquad
 \mathbb W_z^{\rm rep}:=0,\qquad
 \mathbb W_i^{\rm rep}:=\mathbb W_i+c^{\rm rep}
 \quad(i\in\mathcal P),\notag\\
 \sum_i\mu_i\mathbb W_i^{\rm rep}
 &=\sum_i\mu_i\mathbb W_i.
 \label{eq:app-simple-zero-repair}
\end{align}
The axis defect is thereby redistributed among the positive-weight nodes on
the same line.  The repair restores the exact stored axis condition while
preserving the linewise GLL quadrature moment, and hence the conservative
element average.

\paragraph{OE on an axis cell}
After repair, OE is applied to the regular local state
\begin{equation}
 \widehat U_{\boldsymbol\ell}
 =\begin{cases}
 \mathbb W_{\boldsymbol\ell}^{\rm rep}/\chi_{\boldsymbol\ell},
 & \chi_{\boldsymbol\ell}>0,\\[1ex]
 \mathcal T_{\boldsymbol\ell}^{\widehat U}(\mathbb W_h^{\rm rep}),
 & \chi_{\boldsymbol\ell}=0.
 \end{cases}
 \label{eq:app-axis-oe-local-state}
\end{equation}
The jump indicators and damping factors are unchanged from
Section~\ref{sec:algorithm-oe}.  Only the $\chi$-weighted projections require
a convention at the zero-weight face.  For $0\le r\le N-1$,
$P_{\chi,K}^r$ is defined by the orthogonality condition in
\eqref{eq:oe-weighted-shell}, whereas
\begin{equation}
 P_{\chi,K}^N:=\mathrm{Id}.
 \label{eq:app-axis-oe-top-projection}
\end{equation}
For $r\le N-1$, the weighted Gram matrix remains positive definite: a
polynomial in $\mathbb Q_r(K)$ with zero weighted norm must vanish at all $N$
positive-weight GLL nodes on each normal line and therefore vanish
identically.  At $r=N$, the axis-face nodal mode lies in the null space of the
weighted form, which motivates
\eqref{eq:app-axis-oe-top-projection}.

With this convention, the hierarchical increments
\eqref{eq:oe-weighted-shell}, damping factors
\eqref{eq:oe-shell-retention}, and conservative map-back
\eqref{eq:oe-conservative-mapback} apply unchanged.  The top increment also
has zero $\chi$-weighted mean because $P_{\chi,K}^N=\mathrm{Id}$ and
$P_{\chi,K}^{N-1}$ preserves constants.  Hence, writing
$\mathbb W^{\rm OE}:=\mathcal O_K^{\rm cons}\mathbb W^{\rm rep}$,
\begin{equation}
 \sum_{\boldsymbol\ell}\mu_{K,\boldsymbol\ell}
 \mathbb W_{\boldsymbol\ell}^{\rm OE}
 =\sum_{\boldsymbol\ell}\mu_{K,\boldsymbol\ell}
 \mathbb W_{\boldsymbol\ell}^{\rm rep},\qquad
 \mathbb W_{\boldsymbol\ell}^{\rm OE}=0
 \quad\text{if }\chi_{\boldsymbol\ell}=0.
 \label{eq:app-axis-oe-properties}
\end{equation}
Thus OE acts on a finite local state without singular division while
preserving both the conservative element average and the exact stored axis
value.

\paragraph{Trace-aware PCP scaling}
The same trace construction is compatible with PCP scaling because
\eqref{eq:app-simple-zero-trace} is linear in the stored polynomial.  Let
$\widehat U_{A,K}$ and
$\mathbb W_h^{\rm ref}=\chi\widehat U_{A,K}$ be the anchor and reference
polynomial from \eqref{eq:pcp-weight-anchor}.  Then
$\mathcal T_\star^{\widehat U}(\mathbb W_h^{\rm ref})
=\widehat U_{A,K}$ and, for $0\le\vartheta\le1$,
\begin{equation}
 \mathcal T_\star^{\widehat U}
 \left(\mathbb W_h^{\rm ref}
 +\vartheta(\mathbb W_h-\mathbb W_h^{\rm ref})\right)
 =\widehat U_{A,K}
 +\vartheta\left[
 \mathcal T_\star^{\widehat U}(\mathbb W_h)-\widehat U_{A,K}
 \right].
 \label{eq:pcp-simple-zero-segment}
\end{equation}
Thus contraction of the stored polynomial toward the reference polynomial
induces the same convex contraction of the reconstructed axis state toward
$\widehat U_{A,K}$.  The three scaling factors of
Section~\ref{sec:pcp-weighted-scaling} are therefore computed over the
positive-weight nodal states together with the regular axis traces.  Since
both the current and reference polynomials vanish at zero-weight nodes, every
scaling stage also preserves the stored axis constraint.

\begin{theorem}[Conservative PCP scaling on an axis cell]
\label{thm:pcp-simple-zero-scaling}
Assume $\overline\chi_K>0$, $\overline{\mathbb W}_K\in\Gset$, and that every
axis-face zero of $\chi$ is simple in the sense of
\eqref{eq:app-simple-zero-assumption}.  After the repair
\eqref{eq:app-simple-zero-repair}, trace-aware three-stage PCP scaling
preserves $\overline{\mathbb W}_K$ and $\mathbb W_\star=0$ and places every
positive-weight nodal state and regular axis trace in $\Gset_K^{\rm cert}$.
If these states already satisfy the certification conditions, all three
scaling factors equal one.
\end{theorem}
\begin{proof}
The anchor belongs to $\Gset$, and the three convex certification sets contain
it.  Equation~\eqref{eq:pcp-simple-zero-segment} extends each nodal
contraction to the regular axis traces.  Each stage preserves the conservative
average because
$\overline{\chi\widehat U_{A,K}}_K=\overline{\mathbb W}_K$, and it preserves
$\mathbb W_\star=0$ because the current and reference states both vanish
there.  The final claim follows from the definitions of the scaling factors.
\end{proof}